\pdfoutput=1
\documentclass[a4paper]{article}
\usepackage[T1]{fontenc}
\usepackage[utf8]{inputenc}
\usepackage{graphicx}
\usepackage{authblk}
\usepackage{geometry}
\usepackage{setspace}
\usepackage{verbatim}
\usepackage{ragged2e} 
\usepackage[english]{babel}
\usepackage{amsfonts}
\usepackage{cancel}
\usepackage{latexsym}
\usepackage{amsmath}
\usepackage{amssymb}
\usepackage{caption}
\usepackage{subcaption}
\usepackage{bm}
\usepackage{stmaryrd}
\usepackage{multirow}
\usepackage{indentfirst}
\usepackage{xcolor}
\usepackage{csquotes}
\usepackage{enumerate}
\usepackage[utf8]{inputenc}
\usepackage{enumitem}
\usepackage{siunitx}
\usepackage{amsthm}
\usepackage{float}
\newtheorem{theorem}{Theorem}[section]
\newtheorem{Corollary}{Corollary}[theorem]
\newtheorem{lemma}[theorem]{Lemma}
\newtheorem{Problem}{Abstract Problem}
\newtheorem{Assumption}{Assumption}
\newtheorem{remark}{Remark}
\usepackage[sorting=none]{biblatex}
\usepackage{mismath}

\newcommand{\bbeta}{\bm{\beta}}
\newcommand{\bnu}{\bm{\nu}}

\newcommand{\Piti}{{\mathcal{P}_i t_i}}
\newcommand{\Pmiti}{{\mathcal{P}^m_i t_i}}

\newcommand{\chimi}{\chi^{m-1,m}_i}

\newcommand{\Rglo}{\mathcal{R}^\mathup{glo}}
\newcommand{\Rmm}{\mathcal{R}^{m}}
\newcommand{\Rgloi}{\mathcal{R}^\mathup{glo}_i}
\newcommand{\Rmmi}{\mathcal{R}^{m}_i}

\DeclareMathOperator{\Int}{int}
\DeclareMathOperator{\Cl}{cl}
\DeclareMathOperator{\Argmin}{argmin}
\DeclareMathOperator{\Span}{span}

\newenvironment{Abstract}{
        
	\begin{abstract}
}{
        \end{abstract}
}

\newenvironment{keywords}{
        
	\begin{abstract}
}{
        \end{abstract}
}

\title {Constraint Energy Minimizing Generalized Multiscale Finite Element Method for Convection Diffusion Equations with Inhomogeneous Boundary Conditions}
\author[1]{Po Chai WONG}
\author[1]{Eric T. CHUNG}
\author[1]{Changqing YE}
\author[2]{Lina ZHAO}

\affil[1]{\small Department of Mathematics, The Chinese University of Hong Kong, Hong Kong SAR, China. \textit{\{pcwong, tschung, cqye\}@math.cuhk.edu.hk}}
\affil[2]{\small Department of Mathematics, City University of Hong Kong, Hong Kong SAR, China. \textit{linazha@cityu.edu.hk}}
\usepackage{fancyhdr}
\usepackage{hyperref}
\hypersetup{
  setpagesize  = false,
  colorlinks   = true,    
  urlcolor     = blue,    
  linkcolor    = black,   
  citecolor    = black    
}
\begin{document}

\maketitle
\vspace{6pt}

\begin{Abstract}
    
In this paper, we develop the constraint energy minimizing generalized multiscale finite element method (CEM-GMsFEM) for convection-diffusion equations with inhomogeneous Dirichlet, Neumann and Robin boundary conditions, along with high-contrast coefficients. For time independent problems, boundary correctors $\mathcal{D}^m$ and $\mathcal{N}^{m}$ for Dirichlet, Neumann, and Robin conditions are designed. For time dependent problems, a scheme to update the boundary correctors is formulated. Error analysis in both cases is given to show the first-order convergence in energy norm with respect to the coarse mesh size $H$ and second-order convergence in $L^2-$norm, as verified by numerical examples, with which different finite difference schemes are compared for temporal discretization. Nonlinear problems are also demonstrated in combination with Strang splitting. 
\end{Abstract}


\vspace{6pt}
\section{Introduction}
Convection diffusion equation is involved in many physical applications of partial differential equations. Computational difficulty may arise in two-fold: (1) coefficients in high contrast and multiple scales and (2) demanding discretization for a high Péclet number. A lot of multiscale effort has contributed to the problems such as multiscale finite element method \cite{doi:10.1142/S0219876204000071},
variational multiscale method \cite{doi:10.1137/050645646,JOHN20064594,SONG20102226,XIE2021107077}, multiscale discontinuous Galerkin method \cite{KIM20142251,Chung_Leung_2013}, multiscale stablization \cite{chung2018multiscale,CALO2016359}.
In particular, the Generalized Multiscale Finite Element Method (GMsFEM) aims to create multiscale basis to apply Galerkin approximation and it has been applied to an array of partial differential equations \cite{CHUNG201669,doi:10.1137/130926675,CHUNG201454,CHUNG2019298,EFENDIEV2013116,https://doi.org/10.1002/nme.5958,CHEN2020109133}.
Similar to the finite element method, it consists of two stages: first the offline stage where basis functions are generated and used to span the approximation solution manifold; and second the online stage where the actual approximation is found in the generated space. However, due to the complexity of the problem, it is necessary to further reduce the computational cost of the offline stage. 

A spectral decomposition method, Constraint Energy Minimization Generalized Multiscale Finite Element Method (CEM-GMsFEM), is then applied to the scheme to generate a sufficiently large solution manifold by using a few basis functions \cite{chung2023multiscale,cheung2018constraint,CHEUNG2020112960,Chung_2018,FU2020109569,Fu2018ConstraintEM,doi:10.1137/18M1193128,VASILYEVA2019660}.
Each basis function is formed by the eigenfunction in a local spectral problem and captures some information about the medium and velocity, and thereby dependent on the Péclet number. One main important property of such eigenfunction is exponential decay. The span of such basis therefore can capture the cell decaying part of the solution while the non-decaying part is left to be dealt with in the online stage. However, it is only common to see the above results conducted in a homogeneous setting whilst, in practice, inhomogeneous boundary conditions are necessary. We propose a numerical scheme to apply CEM-GMsFEM to solve the convection-diffusion equation with inhomogeneous boundary conditions. There are two versions of CEM-GMsFEM: constrained and relaxed. The relaxed-CEM-GMsFEM is used to develop the error analysis, which takes advantage of the elliptic projection of the solution as a bridge between our multiscale space and the continuous space. An account of relaxed-CEM-GMsFEM applied to some inhomogeneous Boundary Value Problems was given already in \cite{doi:10.1137/21M1459113}. Our goal is to extend this idea to the convection-diffusion equation. We will study three main cases, Dirichlet, Neumann, and Robin conditions. We also give a theoretical account of the convergence analysis and prove the scheme has first-order convergence in energy norm and second-order in $L^2-$norm, as verified by numerical examples.

For non-time independent problems, some work of applying CEM-GMsFEM has been done for parabolic equations \cite{doi:10.1137/18M1193128} on homogeneous conditions. Our goal is to extend our method to the convection-diffusion equation with inhomogeneous boundary conditions. Since we assume the medium and velocity are independent of time, we can reuse the multiscale space in the time-independent case. For time-invariant boundary conditions, the corrector can be pre-computed once. However, for time-variant boundary conditions, the time derivative of the corrector and the boundary conditions need to be taken into account. Therefore, we give a new formulation to update the corrector at each time step in a relaxed CEM-GMsFEM fashion. The error of such approximation is also proved to be exponentially decaying in space. In addition, there are two versions of Backward Euler schemes to compute the next steps of the solution: the diffusion approach (D-approach) and the convection-diffusion approach (CD-approach). The former is less expensive but can be shown to bear a lower accuracy. We compare them via convection diffusion IBVPs and verify that our proposal gives a more accurate result at a higher computational expense. 

Moreover, we introduce a nonlinearity into the IBVPs, which greatly increases the difficulty. A classical approach is to apply the Strang splitting method which considers convection, diffusion, and the nonlinearity terms in separate intermediate steps. This method utilizes the symmetric property of the algorithm to split operators in intermediate timesteps. This has been applied to a variety of problems including parabolic equations \cite{POVEDA2024112796},
diffusion-reaction equations \cite{splitting,GERISCH2002159} 
and diffusion-reaction-advection equations with a homogeneous boundary conditions \cite{10.5555/166337.166347}. 
However, it was numerically tested that the inhomogeneous boundary conditions would drag the overall accuracy of the algorithm \cite{HUNDSDORFER1995191,10.5555/166337.166347,10.1093/imanum/drx047}.
To resolve this, some classical approaches can be found in \cite{10.1093/imanum/drx047,LeVeque1981NumericalMB}.
More recent approaches involve designing a time-dependent boundary corrector in the intermediate steps \cite{doi:10.1137/140994204,doi:10.1137/16M1056250,Einkemmer_2018,doi:10.1137/19M1257081},
which is aligned with our previous sections. In particular, in the classical Strang splitting method applied to convection diffusion equations, the convection and diffusion operators are split and considered separately with their corresponding nonlinear terms \cite{CONNORS2014181}.
Now with our scheme, we can consider them in the same step. We tested our choice of boundary corrector, the same as the previous part, to attain both spatial and temporal convergence. 

The paper is organized as follows: we first give the problem setting and some preliminaries in section 2. The convergence analysis of its application to the time-independent convection diffusion equation is given in section 3, along with numerical results on the Dirichlet, Neumann, and Robin conditions. In section 4, both time-variant and time-invariant IBVPs are presented with analysis and numerical results, along with the comparison of different finite difference schemes for temporal discretization. A demonstration of applying this to nonlinear problems with Strang splitting is presented in this section as well.  
\section{Preliminaries}
\subsection{Problem setting}

We consider the following convection diffusion initial boundary value problem:
\[
\begin{cases}
    \partial_t u + \bbeta(x) \cdot \nabla u = \nabla\cdot(\boldsymbol{A}(x)\nabla u)+f, &\text{ in }\Omega\times (0,T],\\
    u = g(x,t), &\text{ on }\Gamma_D\times (0,T],\\
    b(x) u + \bnu \cdot(\boldsymbol{A}(x)\nabla u -\bbeta(x) u) = q(x,t), &\text{ on }\Gamma_N\times (0,T],\\
    u(\cdot,0)=u_\mathup{init}, &\text{ in }\Omega,
\end{cases}
\]
where $\Omega\subset \mathbb{R}^d$ is the computational domain, $\bbeta\in L^{\infty}(\Omega)^2$, and $0<T<\infty$. 
The medium $\boldsymbol{A}\in L^{\infty}(\Omega;\mathbb{R}^{d\times d})$ and the velocity $\bbeta$ are heterogeneous coefficients with multiple scales and potentially high contrast. $\boldsymbol{A}$ is a positive definite matrix. There exist $\kappa_0$ and $\kappa_1$ such that $0<\kappa_0\leq\lambda_{\min} (\boldsymbol{A}(x))\leq\lambda_{\max}(\boldsymbol{A}(x))\leq\kappa_1$ and $\kappa_1/\kappa_0$ can be large.
We denote $\bnu$ the outward unit normal vectors to $\partial\Omega$ and $\Gamma_D$ and $\Gamma_N$ two nonempty disjoint part of $\partial\Omega$.
We assume the velocity flows inward on the Neumann boundary $\Gamma_N$ and $\bbeta$ is incompressible, i.e.,
$\nabla\cdot\bbeta=0$. i.e., $\bbeta\cdot\bnu \leq 0$ on $\Gamma_N$. Denote $\beta_0\geq 1$ and $\beta_1$ as the infimum and supremum of $|\bbeta|$ respectively. 
The function $b(x)\geq 0$ a.e.\ $x\in\partial\Omega$ and there exists a positive constant $b_0>0$ and a subset $\Gamma \subset \Gamma_N$ with positive measure such that $b(x)\geq b_0$ for a.e.\ $x\in \Gamma$. 
The Dirichlet boundary value term $g\in\mathcal{H}^{1/2}(\Gamma_D\times[0,T])$ and the Neumann boundary value term $q\in L^2(\Gamma_N\times[0,T])$. 

From now on, we denote
\begin{align*}
a(w,v)&=\int_\Omega \boldsymbol{A}\nabla w\cdot \nabla v  + \int_{\Gamma_N}(b-\bbeta\cdot\bnu) wv \di \sigma,\\
\mathcal{A}(w,v)&=\int_\Omega \boldsymbol{A}\nabla w \cdot \nabla v +\int_{\Gamma_N}(b-\bbeta\cdot\bnu) wv \di \sigma + \int_{\Omega}(\bbeta\cdot\nabla w) v . \end{align*}

So the variational form of the problem is: for $t\in (0,T]$, find $u_0(\cdot,t)\in V\coloneqq\{v\in\mathcal{H}^1(\Omega)\colon v=0 \text{ on }\Gamma_D\}$ such that at $t$, 
\[
\begin{aligned}
(\partial_t u_0,v)+\mathcal{A}(u_0,v) &= (f,v) - \mathcal{A}(\widetilde{g},v) + (q,v)_{\Gamma_N} &\text{ for }v\in V\\
(u_0(\cdot,0),v)&=(u_\mathup{init}-\widetilde{g}(\cdot,0),v) &\text{ for }v\in V \end{aligned}
\]
where $\widetilde{g}\in\mathcal{H}^1(\Omega\times[0,T])$ with $\widetilde{g}=g$ on $\Gamma_D$. The solution $u$ to equation \eqref{eqn:BVP} is then $u=u_0+\widetilde{g}$.
\vspace{0.2cm}
\noindent
Denote the energy norms on $V$
\[
\|v\|_{a(\omega)}=\left(\int_\omega \boldsymbol{A}\nabla v\cdot \nabla v  + \int_{\Gamma_N\cap \partial\omega}(b-\bbeta\cdot\bnu) v^2 \di \sigma \right)^{1/2}\text{ and } \|v\|_a = \|v\|_{a(\Omega)}.\]
Also, notice that under the assumptions on  $\bbeta$ and $b$, for $v\in V$,
\begin{align*}
    \mathcal{A}(v,v) &= \int_\Omega \boldsymbol{A}\nabla v\cdot \nabla v + \int_\Omega (\bbeta\cdot \nabla v) v + \int_{\Gamma_N} (b-\bbeta\cdot \bnu ) v^2 \di \sigma\\
    &= \int_\Omega \boldsymbol{A}\nabla v\cdot\nabla v+ \int_{\Gamma_N} \left(b-\frac{1}{2}\bbeta\cdot\bnu\right) v^2 \di \sigma \geq 0.
\end{align*}
So, the following quasi-norms are also well defined on $V$:
\[\|v\|^2_{\mathcal{A}(\omega)} = \int_\omega \boldsymbol{A} \nabla v \cdot \nabla v + \int_{\partial\omega\cap\Gamma_N}(b - \frac{1}{2}\bbeta\cdot\bnu) v^2\di \sigma,\]
for $\omega \subset \Omega$ and $\|v\|_\mathcal{A}=\|v\|_{\mathcal{A}(\Omega)}$.
It is also easy to show that $\|\cdot\|_\mathcal{A}$ and $\|\cdot\|_a$ are equivalent on $V$.

\subsection{Constraint energy minimizing generalized multiscale finite element method}
\subsubsection{Auxiliary basis functions}

We first introduce $\mathcal{T}^H$ to be a conforming partition of $\Omega$ into rectangular elements. 
The resulting elements are called the coarse grids.
Let $H$ be the meshsize of the coarse grid and $N$ be the number of elements. 
For each $K_i\in \mathcal{T}^H$ with $1\leq i \leq N$, we define an oversampled domain $K^m_i$ \cite{doi:10.1137/21M1459113} by
$$K^m_i\coloneqq\Int \left\{\left(\bigcup_{K\in\mathcal{T}^H,
\Cl(K)\cap \Cl(K^{m-1}_i)\neq \emptyset} \Cl(K)\right)\cup \Cl(K^{m-1}_i)\right\}$$
where $\Int(S)$ and $\Cl(S)$ are the interior and the closure of a set $S$. 
We also set $K^0_i\coloneqq K_i$ for the consistency of notations.
Let $N_\nu$ be the number of vertices contained in an element. We can construct a set of Lagrange bases $\{\eta^1_i,\dots, \eta^{N_\nu}_i\}$ of the element $K_i\in\mathcal{T}^H$.
\begin{figure}[ht]
    \centering
    \includegraphics[scale=0.3]{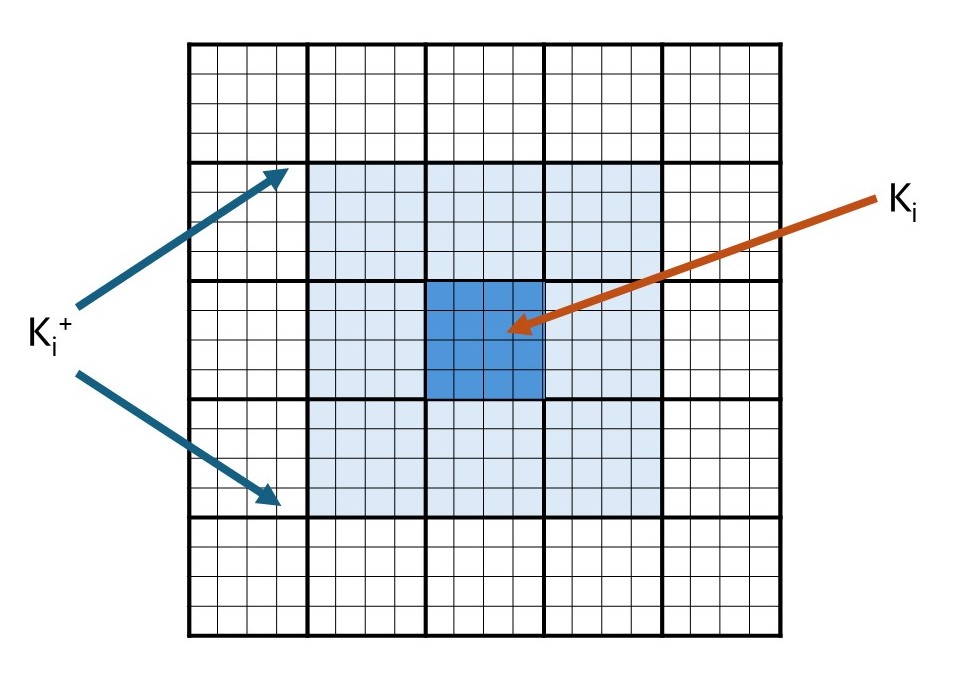}
    \caption{Domain $K_i$ and oversampled domain $K^+_i$}
    \label{fig:domain}
\end{figure}
On each $K_i$, we solve a local spectral problem:
find $\lambda^{j}_i\geq 0$, $\phi^{j}_i\in \mathcal{H}^1(K_i)$
such that for all $v\in\mathcal{H}^1(K_i)$,
\begin{equation}
\mathcal{A}_{(K_i)}(\phi^{j}_i,v)=\lambda^{j}_i s_i(\phi^{j}_i,v),
\label{eq:eigenvalueProblem}\end{equation}
where 
\[ s_i(w,v) = CH^{-2}\int_{K_i} \kappa_1 |\bbeta|^2 wv \coloneqq \int_{K_i}\widetilde{\kappa}wv
\] and $C$ is a constant that depends on the choice of basis $\{\eta_i\}$ and for the local auxiliary space on a structured mesh such that $|\nabla \eta_i|^2 \leq CH^{-2}$.
We take $C = 24$ since we will use the Lagrange polynomials as the basis.  For simplicity, we will omit the constant $C$ in the analysis.
Arrange the eigenvalues $\{\lambda^j_i\}^{\infty}_{j=0}$ in an ascending order.
Define the local auxiliary space $V^\mathup{aux}_i\coloneqq\Span\{\phi^{1}_i,\dots,\phi^{l_i}_i\}$ for some $l_i$. 
Let $s(w,v) = \sum^N_{i=1}s_i(w,v)$ on $V\times V$.
Denote 
$\|w\|_{s(\omega)}=\sqrt{\int_{\omega} \widetilde{\kappa}w^2\di x }$
and 
$\|w\|_{s}=\sqrt{s(w,w)}$.

Now, 
define the orthogonal projection $\pi_i:L^2(K_i)\rightarrow V^\mathup{aux}_i$ by
$$\pi_i(v)\coloneqq\sum_{j=0}^{l_i}\frac{s(\phi^{j}_i,v)}{s(\phi^{j}_i,\phi^{j}_i)}\phi^{j}_i.$$

Let $V^\mathup{aux}\coloneqq\bigoplus^N_{i=1}V^\mathup{aux}_i$. Then $s(\cdot,\cdot)$ and $\|\cdot\|_s$ are an inner product and a norm on $V^\mathup{aux}$ respectively. 
Also $\pi \coloneqq\sum^N_{i=1}\pi_i$ maps $L^2(\Omega)$ to $V^\mathup{aux}$.
Now, we can derive the following lemma \cite{CHUNG2018298}:
\begin{lemma}\label{lemma:eigen_ineq}
    For $v\in\mathcal{H}^1(K_i)$,
    \begin{equation}
    \|v-\pi_i(v)\|_{s(K_i)}^2\leq \frac{\|v\|^2_{a(K_i)}}{\lambda^{l_i+1}_i},
    \label{lemma:eigen}
    \end{equation}
    \begin{equation}
        \|\pi_i(v)\|_{s(K_i)}^2=\|v\|_{s(K_i)}^2
        -\|v-\pi_i(v)\|_{s(K_i)}^2\leq \|v\|_{s(K_i)}^2.
    \label{lemma:auxmap}
    \end{equation}
\end{lemma}
\subsubsection{Multiscale basis functions}
\label{section:multiscale_basis}

Although lemma \ref{lemma:eigen_ineq} has shown that the elliptic projection can approximate any vector $v$ close enough given sufficient number of eigenfunctions used, this function in $V^\mathup{aux}$ may not be continuous in $\Omega$. For this, We now construct the multiscale local basis functions using the auxiliary space $V^\mathup{aux}$. 
Define
\[\psi^{j}_i=\Argmin\left\{\|\psi\|_\mathcal{A}^2+\|\pi\psi-\phi^{j}_i\|_s^2 \colon \psi\in V\right\},\]
\[\psi^{j,m}_i=\Argmin\left\{\|\psi\|_\mathcal{A}^2+\|\pi\psi-\phi^{j}_i\|_s^2 \colon \psi\in V^m_i\right\}.\]
where
\[V^m_i\coloneqq\left\{v\in\mathcal{H}^1(K^m_i)\colon
v=0 \text{ on }\Gamma_D\cap\partial K^m_i\text{ or }\Omega\cap\partial K^m_i\right\}.\]

With these definitions, we can obtain
\begin{equation}
\mathcal{A}(\psi^j_i,v)+s(\pi\psi^j_i,\pi v)=s(\phi^j_i,\pi v),\text{ } \forall v\in V,
\label{eq:psi}\end{equation}
\begin{equation}
\mathcal{A}(\psi^{j,m}_i,v)+s(\pi\psi^{j,m}_i,\pi v)=s(\phi^j_i,\pi v), \text{ }\forall v\in V^m_i.
\label{eq:psi_m}\end{equation}

Denote $V^\mathup{glo}_\mathup{ms}=\Span\{\psi^j_i\colon 0\leq j\leq l_i,1\leq i\leq N\}$ and $V^{m}_\mathup{ms}\coloneqq\Span\{\psi^{j,m}_i\colon 0\leq j\leq l_i,1\leq i\leq N\}$.

\vspace{0.2cm}
\noindent

It is worth-mentioning that the bilinear form $\mathcal{A}$ is not symmetric, and therefore is not an inner product on $V$. However, the following lemma still holds, without stating the proof here for simplicity \cite{CHUNG2018298}.

\begin{lemma}
    Let $v\in V^\mathup{glo}_\mathup{ms}$. Then $\mathcal{A}(v,v^\prime)=0$ for any $v^\prime\in V$ with $\pi v^\prime =0$.
    If there exists $v\in V$ such that $\mathcal{A}(v,v^\prime)=0$ for any $v^\prime\in V^\mathup{glo}_\mathup{ms}$, then $\pi v=0$. 
    \label{lemma:orthogonal}
\end{lemma}

\section{Time-independent convection diffusion boundary value problems}
\label{chapter:1}
As the multiscale space $V^{m}_{ms}$ is independent of $t$, the method is applicable to time independent problems. To illustrate the idea, consider
Find $u\in \mathcal{H}^1(\Omega)$ such that
\begin{equation}
\begin{cases}
    -\nabla \cdot(\boldsymbol{A} \nabla u)+ \bbeta\cdot\nabla u = f, &\text{ in }\Omega,\\
    u=g, &\text{ on }\Gamma_D,\\
    b u + \bnu\cdot (\boldsymbol{A} \nabla u- \bbeta u)=q, &\text{ on }\Gamma_N,
\end{cases}
\label{eqn:BVP}
\end{equation}
where $f\in L^2(\Omega)$ is the source term independent of $u$.
\subsection{Derivation of the Method}
\label{section:method_1}

The methods are the following steps:
\begin{enumerate}
    \item Find $\mathcal{D}^m_i \widetilde{g}\in V^m_i$ and $\mathcal{N}^m_i q\in V^m_i$ such that for all $v\in V^m_i$,
    \[\mathcal{A}(\mathcal{D}^m_i\widetilde{g},v)+s(\pi \mathcal{D}^m_i\widetilde{g},\pi v)=\mathcal{A}_{(K_i)}(\widetilde{g},v),\]
    \[\mathcal{A}(\mathcal{N}_i^m q,v)+s(\pi\mathcal{N}_i^m q,\pi v)=\int_{\partial K_i\cap \Gamma_N}qv \di \sigma.\]
    Further denote $\mathcal{D}^m\widetilde{g}=\sum^N_{i=1}\mathcal{D}^m_i\widetilde{g}$ and 
    $\mathcal{N}^m q=\sum^N_{i=1}\mathcal{N}_i^m q$.

    \item Construct the multiscale function space $V^m_\mathup{ms}$ according to equation \eqref{eq:psi_m}.
    \item Solve $w^m\in V^m_\mathup{ms }$ such that for all $v\in V^m_\mathup{ms}$,
    \begin{equation}
        \mathcal{A}(w^m,v)=(f,v)-\mathcal{A}(\widetilde{g},v)+\int_{\Gamma_N}qv \di \sigma
        +\mathcal{A}(\mathcal{D}^m\widetilde{g},v)-\mathcal{A}(\mathcal{N}^m q,v).
        \label{eq:w^m}
    \end{equation}

    \item Construct the numerical solution $u^\mathup{ms}$ to approximate the actual solution $u$ of equation \eqref{eqn:BVP} by
    \[u^\mathup{ms}_0 = w^m-\mathcal{D}^m\widetilde{g}+\mathcal{N}^m q\text{  and  } u^\mathup{ms}\approx u^\mathup{ms}_0+\widetilde{g}.\]
\end{enumerate}

\subsection{Analysis}
\label{section:analysis_1}

We will use the following notation:
for $w,v\in V$ and $\omega \subset \Omega$,
\[\mathcal{B}_{(\omega)}(w,v) = \mathcal{A}_{(\omega)}(w,v)+s_{(\omega)}(\pi w,\pi v)\]
and 
$\|v\|_{\mathcal{B}(\omega)} \coloneqq \sqrt{\mathcal{B}_{(\omega)}(v,v)}$.
Before we give an account of the analysis, we will summarize all the quasi-norms in the following lemma,
\begin{lemma} For any $v\in V$,
    \[
    \|v\|_{L^2}\leq Hc_\#\|v\|_\mathcal{B},
    \]
    \[
    \mathcal{A}(w,v)\leq \overline{C}\|w\|_\mathcal{A}\|v\|_\mathcal{B},
    \]
    \[
    \mathcal{B}(w,v)\lesssim\overline{C}\|w\|_\mathcal{B}\|v\|_\mathcal{B},
    \]
    where $\Lambda =\min_{i}\lambda^{l_i+1}_i$, $
    c_\# \coloneqq \beta_0^{-1}\kappa_1^{-1/2}\sqrt{1+\Lambda^{-1}}
    \text{, }\overline{C}\coloneqq \sqrt{2}(1+\Lambda^{-1/2}) \max\left\{1, \frac{H}{\sqrt{C\kappa_0\kappa_1}}\right\}$\text{ and }$\lesssim$ is "not greater than up to a positive constant".
    \label{lemma:overline_C}
\end{lemma}
\begin{proof}
    First,
    \begin{align*}
        \|v\|_{L^2}\leq H\beta_0^{-1}\kappa_1^{-1/2}\|v\|_s
        \leq H\beta_0^{-1}\kappa_1^{-1/2}(1+\Lambda^{-1})^{1/2}\|v\|_\mathcal{B}.
    \end{align*}
    On the other hand,
    \begin{align*}
        \mathcal{A}(w,v) &= \int\boldsymbol{A}\nabla w\cdot\nabla v + \int\bbeta\cdot\nabla w v + \int_{\partial \Omega} (b-\frac{1}{2}\bbeta\cdot\bnu)wv\di \sigma\\
        &\leq \|w\|_a \|v\|_a + H\sqrt{\int \boldsymbol{A}\nabla w \cdot\nabla w}
        \sqrt{\int \widetilde{\kappa}v^2}\\
        &\leq \|w\|_a\|v\|_a + H\|w\|_a(\|\pi v\|_s + \|v-\pi v\|_s)\\
        &\leq \|w\|_a\|v\|_a + H\|w\|_a (\|\pi v\|_s + \Lambda^{1/2}\|v\|_a)\\
        &\leq \overline{C}\|w\|_\mathcal{A}\|v\|_\mathcal{B}.
    \end{align*}
    Finally,
    \begin{align*}
        \mathcal{B}(w,v) &= \mathcal{A}(w,v)+ s(\pi w, \pi v)\\
        &\leq \overline{C}\|w\|_\mathcal{A}\|v\|_\mathcal{B}+\|\pi w\|_s \|\pi v\|_s\\
        &\lesssim \overline{C}\|w\|_\mathcal{B}\|v\|_\mathcal{B}.
    \end{align*}
\end{proof}

With the above lemmas, we can now start our error analysis.

\subsubsection{Global Approximations}

We will approximate the error using the global approximation $w^\mathup{glo}$ of the real solution via $V^\mathup{glo}_\mathup{ms}$.
Define $\mathcal{D}^\mathup{glo}\widetilde{g}=\sum^N_{i=1}\mathcal{D}^\mathup{glo}_i\widetilde{g}$ and $\mathcal{N}^\mathup{glo} q=\sum^N_{i=1}\mathcal{N}^\mathup{glo}_i q$ where $\mathcal{D}^\mathup{glo}_i\widetilde{g}\in V^\mathup{glo}_\mathup{ms}$ satisfies that for all $v\in V$,
\begin{equation}\mathcal{B}(\mathcal{D}^\mathup{glo}_i\widetilde{g},v) = \mathcal{A}(\widetilde{g},v), 
\label{eq:Dgtg}\end{equation}
and $\mathcal{N}^\mathup{glo}_i q\in V^\mathup{glo}_\mathup{ms}$ satisfies that for all $v\in V$,
\begin{equation}\mathcal{B}(\mathcal{N}^\mathup{glo}_i q,v)=\int_{\partial K_i\cap \Gamma_N} q v \di \sigma.
\label{eq:Ngq}\end{equation}

Also define $w^\mathup{glo}\in V^\mathup{glo}_\mathup{ms}$ such that,
\begin{equation}
    \mathcal{A}(w^\mathup{glo},v) - \mathcal{A}(\mathcal{D}^\mathup{glo}\widetilde{g}, v) + \mathcal{A}(\mathcal{N}^\mathup{glo} q, v) = \mathcal{A}(u_0,v) \text{ for } v\in V^\mathup{glo}_\mathup{ms};
    \label{eq:w^glo_u0}
\end{equation}

We first show the convergence of the elliptic projection defined using $V^\mathup{glo}_\mathup{ms}$.
\begin{theorem}
    Let $\mathcal{D}^\mathup{glo}_i\widetilde{g}$, $\mathcal{N}^\mathup{glo}_i q$ and $w^\mathup{glo}$ defined by equations \eqref{eq:Dgtg}, \eqref{eq:Ngq} and \eqref{eq:w^glo_u0}. 
    Let $u$ be the actual solution of the problem \eqref{eqn:BVP}. 
    Let $\widetilde{u}_0$ $=w^\mathup{glo}$ $-\mathcal{D}^\mathup{glo}\widetilde{g}$ $+\mathcal{N}^\mathup{glo} q$ and $\widetilde{u} = \widetilde{u}_0+\widetilde{g}$. 
    Then
    \begin{equation}
        \|\widetilde{u}-u\|_\mathcal{A}\leq\Lambda^{-1/2}\kappa_1^{-1/2}H\left(\||\bbeta|^{-1}f\|_{L^2(\Omega)}+\|\nabla u_0\|_{L^2(\Omega)}\right).
        \label{eq:global bound}
    \end{equation}
    \label{thm:1}
\end{theorem}

\begin{proof}
    Let $e=u-\widetilde{u}$. By the definition of $\widetilde{u}$, we have $\mathcal{A}(u-\widetilde{u},v)=\mathcal{A}(u_0-\widetilde{u}_0,v)=0$ for $v\in V^\mathup{glo}_\mathup{ms}$, giving $\pi e = \pi(u-\widetilde{u})=0$ by equations \eqref{eq:Dgtg} and \eqref{eq:Ngq}. This leads to
    \[\mathcal{A}(\mathcal{D}^\mathup{glo}\widetilde{g},e)=\mathcal{A}(\widetilde{g},e) \text{ and } \mathcal{A}(\mathcal{N}^\mathup{glo} q,e)=\int_{\Gamma_N}qe\di \sigma.\]
    Since $\widetilde{w}\in V^\mathup{glo}_\mathup{ms}$, then $\mathcal{A}(\widetilde{w},v) =0$ for $\pi v =0$.
    \begin{align*}
        \|e\|_\mathcal{A}^2 =\mathcal{A}(e,e) 
        &= -\mathcal{A}(\widetilde{w},e) + \mathcal{A}(\mathcal{D}^\mathup{glo}\widetilde{g},e) - \mathcal{A}(\mathcal{N}^\mathup{glo} q,e) + \mathcal{A}(u_0,e)\\
        &= \mathcal{A}(\mathcal{D}^\mathup{glo}\widetilde{g},e) - \mathcal{A}(\mathcal{N}^\mathup{glo} q,e) + \left\{ \int_\Omega fe- \mathcal{A}(\widetilde{g},e)-\int_\Omega\bbeta\cdot\nabla u_0 e\right.
        \left.+\int_{\Gamma_N} qv \di \sigma\right\}\\
        &=  \int_\Omega fe - \int_\Omega \bbeta\cdot\nabla u_0 e\\
        &\leq \|\widetilde{\kappa}^{-1/2}f\|_{L^2(\Omega)}\|e\|_s
        +\kappa_1^{-1/2}H \|\nabla u_0\|_{L^2(\Omega)}\|e\|_s\\
        &\leq \Lambda^{-1/2}\left(\|\widetilde{\kappa}^{-1/2}f\|_{L^2(\Omega)}+\kappa_1^{-1/2}H \|\nabla u_0\|_{L^2(\Omega)}\right)\|e\|_\mathcal{A}.
    \end{align*}

    Hence, we obtain
    \[\|u-\widetilde{u}\|_\mathcal{A} \leq \Lambda^{-1/2}\kappa_1^{-1/2}H\left(\||\bbeta|^{-1}f\|_{L^2(\Omega)}+\|\nabla u_0\|_{L^2(\Omega)}\right).\]
\end{proof}
\noindent
\subsubsection{Abstract Problem}

Following the approach in \cite{doi:10.1137/21M1459113}, we summarise the analysis of CEM-GMsFEM by considering the following abstract problem:
\begin{Problem}
    Let $K_i\in \mathcal{T}^H$ and $t_i\in V^\prime$ such that $\left<t_i,v\right>=0$ for any $v\in V$ with $supp(v)\subset \Omega\backslash K_i$.
    Define $\mathcal{P}_i:V^\prime\rightarrow V$ such that for all $v\in V$,
    \begin{equation}
        \mathcal{B}(\Piti,v)=\left<t_i,v\right>
        \label{eq:piti}
    \end{equation}
    and $\mathcal{P}^m_i:V^\prime\rightarrow V^m_i$ with
    \begin{equation}
        \mathcal{B}(\Pmiti,v)=\left<t_i,v\right>.
        \label{eq:pmiti}
    \end{equation}

We aim to estimate 
    \[\left\|\sum^N_{i=1}\Piti-\Pmiti\right\|_\mathcal{B}^2=\left\|\sum^N_{i=1}\Piti-\Pmiti\right\|_\mathcal{A}^2
    +\left\|\sum^N_{i=1}\Piti-\Pmiti\right\|_s^2.\]
\end{Problem}

To solve this problem, we prepare ourselves with the following lemmas \ref{lemma:part1}, \ref{lemma:part2} and \ref{lemma:part3}. 
To prove them, we need to define cutoff functions $\{\chi_i^{n,m}\}$.
Let $V^H$ be the Lagrange basis function space of $\mathcal{T}^H$. For $K_i\in\mathcal{T}^H$, 
a cutoff function $\chi^{n,m}_i\in V^H$ with $n<m$ satisfies that:
\[\chi^{n,m}_i(x)=1\text{ in }K^n_i;\]
\[\chi^{n,m}_i(x)=0\text{ in }\Omega\backslash K^m_i;\]
\[0\leq\chi^{n,m}_i\leq 1\text{ in }K^m_i\backslash K^n_i.\]

\begin{lemma}
    Let $m\geq 1$. Then there exists $0<\theta<1$ such that
    \[\|\Piti\|_{\mathcal{B}(\Omega\backslash K^m_i)}^2
    \leq \theta^m\|\Piti\|_{\mathcal{B}}^2\]
    where 
    $\theta=\frac{c_*}{c_*+1}$
    and
    \[c_*(\Lambda,\beta_0)=
    \max_{x\in[0,\pi/2]}((\beta_0^{-1}+\kappa_0^{-1/2}\kappa_1^{-1/2}H)\cos(x)+\sin(x))
    \left(\Lambda^{-1/2}\cos(x)+\sin(x)\right).\]
    \label{lemma:part1}
\end{lemma}
\begin{proof}
    
    Note that $1-\chimi\equiv 0$ in $K^{m-1}_i$
    and $1-\chimi\equiv 1$ in $\Omega\backslash K_i$,
    implying that $supp(1-\chimi)\subset \Omega\backslash K_i^{m-1}$.
    Put $v\coloneqq\left(1-\chimi\right)\Piti$ in equation \eqref{eq:piti}, namely
    \begin{align*}
        \mathcal{A}(\Piti,(1-\chimi)\Piti)+ s(\pi\Piti,\pi(1-\chimi)\Piti) &= 0.\\
        \end{align*}
    Then,
    \begin{align*}
        &\quad \|\Piti\|_{\mathcal{A}(\Omega\backslash K^m_i)}^2+\|\pi\Piti\|_{s(\Omega\backslash K^m_i)}^2 \\
        &=
        \int_{K^m_i\backslash K^{m-1}_i}\left(\chimi-1\right)\boldsymbol{A} \nabla\Piti\cdot\nabla\Piti \di x+\int_{K^m_i\backslash K^{m-1}_i}\Piti\boldsymbol{A}\nabla\Piti\cdot\nabla\chimi \di x\\
        &\quad +\int_{K^m_i\backslash K^{m-1}_i} H^{-2}\kappa_1|\bbeta|^2 \pi(\Piti)\cdot
        \pi\left(\left(\chimi-1\right)\Piti\right) \di x +\int_{\Gamma_N\backslash\partial K^m_i}(-b+\bbeta\cdot\bnu) \Piti^2 \\
        &\quad + \int_{K^m_i\backslash K^{m-1}_i}\bbeta\cdot\nabla \Piti (\chimi-1)\Piti\\
        &\eqqcolon I_1 + I_2 + I_3 + I_4 + I_5.
    \end{align*}
    Since $\chimi-1\leq 0$ in $K^m_i\backslash K^{m-1}_i$, then $I_1\leq 0$.
    Also, since $\bbeta\cdot\bnu < b$ on $\Gamma_N$, we have $I_4\leq 0$.\\
    For $I_2$, by using the Cauchy Schwartz inequality,
    \begin{align*}
        &\quad \int_{K^m_i\backslash K^{m-1}_i}\Piti\boldsymbol{A}\nabla\Piti\cdot\nabla\chimi\\
        &\leq \sqrt{\int_{K^m_i\backslash K^{m-1}_i}\boldsymbol{A}\nabla\Piti\cdot\nabla\Piti} \, \sqrt{\int_{K^m_i\backslash K^{m-1}_i}\left(\Piti\right)^2 \boldsymbol{A}\nabla\chimi\cdot\nabla\chimi \, \di x}\\
        &\leq \beta_0^{-1}\left\|\Piti\right\|_{a(K^m_i\backslash K^{m-1}_i)}
        \|\Piti\|_{s(K^m_i\backslash K^{m-1}_i)}.
\end{align*}
    The last inequality is due to the assumptions on $\boldsymbol{A}$ and $\bbeta$,
    \begin{align*}
    \beta_0^2\int_{K^m_i\backslash K^{m-1}_i} \Piti^2 \boldsymbol{A}\nabla\chimi\cdot\nabla\chimi
    \leq
    \int_{K^m_i\backslash K^{m-1}_i} \Piti^2 |\bbeta|^2 CH^{-2}\kappa_1
    = \|\Piti\|_{s(K^m_i\backslash K^{m-1}_i)}^2.    
    \end{align*}
    For $I_3$,
    \begin{align*}
        &\quad \int_{K^m_i\backslash K^{m-1}_i}H^{-2}\kappa_1|\bbeta|^2 \pi(\Piti)\cdot
        \pi\left(\left(\chimi-1\right)\Piti\right) \di x\\
        &\leq \|\pi\Piti\|_{s(K^m_i\backslash K^{m-1}_i)}\|\pi((\chimi-1)\Piti)\|_{s(K^m_i\backslash K^{m-1}_i)}\\
        &\leq \|\pi\Piti\|_{s(K^m_i\backslash K^{m-1}_i)}\|(\chimi-1)\Piti\|_{s(K^m_i\backslash K^{m-1}_i)}\\
        &\leq \|\pi\Piti\|_{s(K^m_i\backslash K^{m-1}_i)}\|\Piti\|_{s(K^m_i\backslash K^{m-1}_i)}.
    \end{align*}
    For $I_5$,
    \begin{align*}
        &\quad \int_{K^m_i\backslash K^{m-1}_i} \bbeta\cdot\nabla\Piti (\chimi-1)\Piti\\
        &\leq \kappa_0^{-1/2}\kappa_1^{-1/2}H\sqrt{\int_{K^m_i\backslash K^{m-1}_i} \boldsymbol{A}\nabla \Piti\cdot\nabla\Piti} \sqrt{\int_{K^m_i\backslash K^{m-1}_i}|\bbeta|^2 H^{-2}\kappa_1 |\Piti|^2}\\
        &\leq \kappa_0^{-1/2}\kappa_1^{-1/2}H \|\Piti\|_{a(K^m_i\backslash K^{m-1}_i)}\|\Piti\|_{s(K^m_i\backslash K^{m-1}_i)}.
    \end{align*}
    By inequality \eqref{lemma:eigen},
    \begin{align*}
        \|\Piti\|_{s(K^m_i\backslash K^{m-1}_i)}
        &\leq \|\Piti-\pi\Piti\|_{s(K^m_i\backslash K^{m-1}_i)}+\|\pi\Piti\|_{s(K^m_i\backslash K^{m-1}_i)}\\
        &\leq \Lambda^{-1/2}\|\Piti\|_{a(K^m_i\backslash K^{m-1}_i)}+\|\pi\Piti\|_{s(K^m_i\backslash K^{m-1}_i)}.
    \end{align*}
    So, in summary,
    \begin{align*}
        \|\Piti\|_{\mathcal{B}(\Omega\backslash K^m_i)}^2
        &\leq \left((\beta_0^{-1}+\kappa_0^{-1/2}\kappa_1^{-1/2}H)\left\|\Piti\right\|_{a(K^m_i\backslash K^{m-1}_i)}+\|\pi\Piti\|_{s(K^m_i\backslash K^{m-1}_i)}\right)\|\Piti\|_{s(K^m_i\backslash K^{m-1}_i)}\\
        &\leq c_*(\Lambda,\beta_0)\|\Piti\|^2_{\mathcal{B}{(K^m_i\backslash K^{m-1}_i)}}.
    \end{align*}
    The last inequality comes from considering
    \[\cos(\theta) = \frac{\|\Piti\|_{\mathcal{A}(K^m_i\backslash K^{m-1}_i)}}{\|\Piti\|_{\mathcal{B}(K^m_i\backslash K^{m-1}_i)}} \text{ and }\sin(\theta)=\frac{\|\pi\Piti\|_{s(K^m_i\backslash K^{m-1}_i)}}{\|\Piti\|_{\mathcal{B}(K^m_i\backslash K^{m-1}_i)}}\]
    for some $\theta\in(0,\pi/2]$. Now,
    \begin{align*}
        \|\Piti\|^2_{\mathcal{B}{(K^{m-1}_i)}}
        &= \|\Piti\|^2_{\mathcal{B}(K^m_i\backslash K^{m-1}_i)}+\|\Piti\|^2_{\mathcal{B}(\Omega\backslash K^{m}_i)}\\
        &\geq \left(1+\frac{1}{c_*}\right)
        \|\Piti\|_{\mathcal{B}(\Omega\backslash K^{m-1}_i)}^2.
    \end{align*}
    Then, iteratively we can obtain
    \[\|\Piti\|_{\mathcal{B}(\Omega\backslash K^m_i)}^2
    \leq \theta^m\|\Piti\|_{\mathcal{B}}^2.\]
\end{proof}
\begin{lemma}
    With the notations in lemma \ref{lemma:part1},
    then
    \[\|\Piti-\Pmiti\|_{\mathcal{B}}^2
    \leq \overline{C}^2 c_\star\theta^{m-1}\|\Piti\|_{\mathcal{B}}^2,\]
    where
    \begin{align*}&c_\star(\Lambda,\beta_0)
    \coloneqq\max_{x\in[0,\pi/2]}
    \left[(1+\beta^{-1}_0\Lambda^{-1/2})\cos(x)
    +\beta^{-1}_0\sin(x)\right]^2
    +\left[\Lambda^{-1/2}\cos(x)+\sin(x)\right]^2.
    \end{align*}
    \label{lemma:part2}
\end{lemma}
\begin{proof}
    Let $z_i\coloneqq\Piti-\Pmiti$ and decompose it as 
    \[z_i=\left\{\left(1-\chimi\right)\Piti\right\}
    +\left\{\left(\chimi-1\right)\Pmiti+\chimi z_i\right\}
    \eqqcolon z^\prime_i+z^{\prime\prime}_i.\]
    By definition, $z^{\prime\prime}_i\in V^m_i$, so by equations \eqref{eq:piti} and \eqref{eq:pmiti},
    $\mathcal{B}(z_i,z_i^{\prime\prime})=0$.
    Then,
    \begin{align*}
        \|z_i\|_\mathcal{B}^2
        =\mathcal{B}(z_i,z_i^\prime)
        \leq \overline{C}\|z_i\|_\mathcal{B}\|z_i^\prime\|_{\mathcal{B}}
        \leq \overline{C}^2 \|z_i^\prime\|_\mathcal{B}^2.
    \end{align*}
    To compute $\|z^\prime_i\|_\mathcal{B}$, we investigate $\|z^\prime_i\|_\mathcal{A}$ and $\|z_i^\prime\|_s$.
    \begin{align*}
        \|z_i^\prime\|_\mathcal{A}^2
        &=\|\left(1-\chimi\right)\Piti\|_\mathcal{A}^2 \lesssim \|\left(1-\chimi\right)\Piti\|_a^2\\
        &\leq \int_{\Omega\backslash K^{m-1}_i} (1-\chimi)^2 \boldsymbol{A}\nabla \Piti\cdot\nabla\Piti - 2\int_{\Omega\backslash K^{m-1}_i} \Piti (1-\chimi)\boldsymbol{A}\nabla\Piti\cdot\nabla\chimi \\
        &\quad +\int_{K^m_i\backslash K^{m-1}_i}\Piti^2 \boldsymbol{A}\nabla\chimi\cdot\nabla\chimi +\int_{\Gamma_N\cap \partial K^m_i} (b-\bbeta\cdot\bnu) \Piti^2 \di \sigma \\
        &\leq \beta^{-2}_0\|\Piti\|_{s(K^m_i\backslash K^{m-1}_i)}^2
        +2\beta^{-1}_0\|\Piti\|_{s(\Omega\backslash K^{m-1}_i)}\|\Piti\|_{a(\Omega\backslash K^{m-1}_i)}
        +\|\Piti\|_{a(\Omega\backslash K^{m-1}_i)}^2\\
        &=\left(\beta^{-1}_0\|\Piti\|_{s(K^m_i\backslash K^{m-1}_i)}+
        \|\Piti\|_{a(\Omega\backslash K^{m-1}_i)}\right)^2\\
        &=\left(\beta^{-1}_0\|\Piti\|_{s(K^m_i\backslash K^{m-1}_i)}+
        \|\Piti\|_{\mathcal{A}(\Omega\backslash K^{m-1}_i)}\right)^2.
    \end{align*}
    Again, by inequality \eqref{lemma:eigen},    
    \begin{align*}
        \|z_i^\prime\|_\mathcal{A}&\leq \beta^{-1}_0\|\Piti\|_{s(K^m_i\backslash K^{m-1}_i)}+
        \|\Piti\|_{\mathcal{A}(\Omega\backslash K^{m-1}_i)}\\
        &=\left(\beta^{-1}_0\Lambda^{-1/2}+1\right)\|\Piti\|_{\mathcal{A}(\Omega\backslash K^{m-1}_i)}+\beta^{-1}_0\|\pi \Piti\|_{s(\Omega\backslash K^{m-1}_i)}.
    \end{align*}
    Also,
    \begin{align*}
        \|\pi z_i^\prime\|_s^2&=
        \left\|\pi\left(\left(1-\chimi\right)\Piti\right)\right\|_s \leq\left\|\left(1-\chimi\right)\Piti\right\|_s\\
        &\leq \left\|\Piti\right\|_{s(\Omega\backslash K^{m-1}_i)} \leq \Lambda^{-1/2}\|\Piti\|_{\mathcal{A}(\Omega\backslash K^{m-1}_i)}
        +\|\pi \Piti\|_{s(\Omega\backslash K^{m-1}_i)}.
    \end{align*}
    Therefore, we can obtain
    \begin{align*}
        \|z_i\|_{\mathcal{A}}^2+\|\pi z_i\|_s^2
        &\leq \overline{C}^2 \|z_i^\prime\|_\mathcal{B}^2\\
        &\leq \overline{C}^2\left\{\left[\left(\beta^{-1}_0\Lambda^{-1/2}+1\right)\|\Piti\|_{\mathcal{A}(\Omega\backslash K^{m-1}_i)}+\beta^{-1}_0\|\pi \Piti\|_{s(\Omega\backslash K^{m-1}_i)}\right]^2 \right.\\
        &\quad +\left.(\Lambda^{-1/2}\|\Piti\|_{\mathcal{A}(\Omega\backslash K^{m-1}_i)}
        +\|\pi \Piti\|_{s(\Omega\backslash K^{m-1}_i)})^2\right\}\\
        &\leq \overline{C}^2 c_\star \|\Piti\|_{\mathcal{B}(\Omega\backslash K^{m-1}_i)}^2.
    \end{align*}
    Thus by lemma \ref{lemma:part1},
    \begin{align*}
     \|\Piti-\Pmiti\|_\mathcal{B}^2
    \leq \overline{C}^2 c_\star\theta^{m-1}\|\Piti\|_\mathcal{B}^2.
    \end{align*}
\end{proof}
\begin{Assumption}
    There exists a constant $C_\mathup{ol}>0$ such that for all $K_i\in\mathcal{T}^H$ and $m>0$,
    \[\#\{K\in\mathcal{T}^H\colon K\subset K^m_i\}
    \leq C_\mathup{ol}m^d.\]
\end{Assumption}

\begin{lemma}
With the notations in lemmas \ref{lemma:part1} and \ref{lemma:part2}, then
\[
\begin{aligned}
    \left\|\sum^N_{i=1}\Piti-\Pmiti\right\|_\mathcal{B}^2\leq \overline{C}^4 C_\mathup{ol} c_\star^3 (m+1)^d \theta^{m-1}\sum^N_{i=1}\left<t_i,\Piti\right>.
\end{aligned}\]
\label{lemma:part3}
\end{lemma}
\begin{proof}
    Let $z_i\coloneqq\Piti-\Pmiti$ and $z\coloneqq\sum^N_{i=1}z_i$.
    Decompose $z$ as
    \[z=\left\{\left(1-\chi^{m,m+1}_i\right)z\right\}
    +\left\{\chi^{m,m+1}_i z\right\}=:z^\prime+z^{\prime\prime}.\]
    Notice that $supp(z^\prime)\subset \Omega\backslash K^m_i$ and thus $supp(\pi z^\prime)\subset \Omega\backslash K^m_i$. 
    Also, $supp(\Pmiti)\subset cl(K^m_i)$ and thus $supp(\pi\Pmiti)\subset cl(K^m_i)$.
    Then, by equations \eqref{eq:piti} and \eqref{eq:pmiti},
    \[
        \mathcal{B}(\Pmiti,z^\prime)=0 \text{ and }\mathcal{B}(\Piti,z^\prime)=0,
    \]
    granting us \[\mathcal{B}(z_i,z^\prime)=0.\]
    Now, similar to the proof of lemma \ref{lemma:part2},
    \begin{align*}
        \mathcal{B}(z_i,z) & = \mathcal{B}(z_i,z^{\prime\prime}) \leq \overline{C} \|z_i\|_\mathcal{B}\|z_i^{\prime\prime}\|_\mathcal{B}\\
        &\leq \overline{C} \left\{ \left((1+\beta_0^{-1}\Lambda^{-1/2})\|z\|_{\mathcal{A}(K^{m+1}_i)}+\beta_0^{-1}\|\pi z\|_{s(K^{m+1}_i)}\right)^2\right.\\
        &\quad \left. + \left(\Lambda^{-1/2}\|z\|_{\mathcal{A}(K^{m+1}_i)} + \|\pi z\|_{s(K^{m+1}_i)}\right)^2\right\}^{1/2}\|z_i\|_\mathcal{B}\\
        &\leq \overline{C} c_\star \|z\|_{\mathcal{B}{(K^{m+1}_i)}}\|z_i\|_\mathcal{B}.
    \end{align*}
    Also, by the definition of $C_\mathup{ol}$,
    \[\sum^N_{i=1}\|z\|_{\mathcal{A}(K^{m+1}_i)}^2+\|\pi z\|_{s(K^{m+1}_i)}^2
    \leq C_\mathup{ol}(m+1)^d\|z\|_\mathcal{B}^2,\]
    and recall by equation \eqref{eq:piti},
    \[\|\Piti\|_\mathcal{A}^2+\|\pi\Piti\|_s^2=\left<t_i,\Piti\right>.\]
    Hence, by the Cauchy--Schwartz inequality and lemma \ref{lemma:part2},
    \begin{align*}
        \|z\|_\mathcal{B}^2&=\|z\|_\mathcal{A}^2 + \|\pi z\|_s^2 =\sum^N_{i=1}\mathcal{B}(z_i,z) \leq \sum^N_{i=1}\overline{C}c_\star \|z\|_{\mathcal{B}(K^{m+1}_i)}\|z_i\|_\mathcal{B}\\
        &\leq\overline{C}c_\star\left[\sum^N_{i=1}\|z\|_{\mathcal{B}(K^{m+1}_i)}^2\right]^{1/2}
        \left[\sum^N_{i=1}\|z_i\|_\mathcal{B}^2\right]^{1/2}\\
        &\leq \overline{C}c_\star \left[C_\mathup{ol}(m+1)^d\|z\|_{\mathcal{B}(K^{m+1}_i)}^2\right]^{1/2}
        \left[\sum^N_{i=1}\|z_i\|_\mathcal{B}^2\right]^{1/2}\\
        &\leq \overline{C}c_\star \left[ C_\mathup{ol}(m+1)^d\|z\|_{\mathcal{B}}^2\right]^{1/2}
        \left[\overline{C}^2 c_\star \theta^{m-1}\sum^N_{i=1}\left<t_i,\Piti\right>\right]^{1/2}\\
        &\leq \overline{C}^4 C_\mathup{ol} c_\star^3 (m+1)^d \theta^{m-1}\sum^N_{i=1}\left<t_i,\Piti\right>.
    \end{align*}
\end{proof}
\subsubsection{Error Estimates for Boundary Correctors}

We now directly apply these results to estimate $\mathcal{D}^\mathup{glo}\widetilde{g}$ $-\mathcal{D}^m\widetilde{g}$ and $\mathcal{N}^\mathup{glo} q$ $-\mathcal{N}^m q$.
\begin{Corollary}
    With the notations in lemmas \ref{lemma:part1}, \ref{lemma:part2} and \ref{lemma:part3}, 
    \begin{equation}
    \begin{aligned}
        \|\mathcal{D}^\mathup{glo}\widetilde{g}-\mathcal{D}^m\widetilde{g}\|_\mathcal{B}^2
        \leq  \overline{C}^5 C_\mathup{ol} c_\star^3 (m+1)^d \theta^{m-1}\|\widetilde{g}\|^2_{\mathcal{A}}  ;
        \end{aligned}
        \label{eq:dgtg-dmtg}
    \end{equation}
    \begin{equation}
        \|\mathcal{N}^\mathup{glo} q-\mathcal{N}^m q\|_\mathcal{B}^2
        \leq \overline{C}^4 C_\mathup{ol} c_\star^3 (m+1)^d \theta^{m-1} C^2_\mathup{tr} \|q\|_{L^2(\Gamma_N)}^2,
        \label{eq:ngq-nmq}
    \end{equation}
    where \[C_\mathup{tr}\coloneqq\sup_{v\in V,v\neq 0}\frac{\|v\|_{L^2(\Gamma_N)}}{\|v\|_\mathcal{A}}.\]
    
    \label{corollary:1}
\end{Corollary}
\begin{proof}
    With lemma \ref{lemma:part3}, it suffices to estimate
    \[\sum^N_{i=1}\mathcal{A}_{(K_i)}(\widetilde{g},\mathcal{D}^\mathup{glo}_i\widetilde{g})\text{ and } 
    \sum^N_{i=1}\int_{\Gamma_N\cap \partial K_i}q \mathcal{N}^\mathup{glo}_i q \di \sigma.\]
    \noindent
   Put $v=\mathcal{D}^\mathup{glo}_i\widetilde{g}$ in equation \eqref{eq:Dgtg} and since $\mathcal{D}^\mathup{glo}_i\widetilde{g}\in V^\mathup{glo}_\mathup{ms}$,
    \begin{align*}
        \|\mathcal{D}^\mathup{glo}_i\widetilde{g}\|_\mathcal{B}^2 = \mathcal{A}_{(K_i)}(\widetilde{g},\mathcal{D}^\mathup{glo}_i\widetilde{g}) \leq \overline{C}\|\widetilde{g}\|_{\mathcal{A}(K_i)}\|\mathcal{D}^\mathup{glo}_i\widetilde{g}\|_\mathcal{B} \leq \overline{C}^2 \|\widetilde{g}\|_{\mathcal{A}(K_i)}^2.
    \end{align*}
    Now, by lemma \ref{lemma:part3},
    \begin{align*}
        \|\mathcal{D}^\mathup{glo}\widetilde{g}-\mathcal{D}^m\widetilde{g}\|_\mathcal{B}^2&\leq \overline{C}^4 C_\mathup{ol} c_\star^3 (m+1)^d \theta^{m-1}\sum^N_{i=1}\mathcal{A}_{(K_i)}(\widetilde{g},\mathcal{D}^\mathup{glo}_i\widetilde{g})\\
        &\leq \overline{C}^5 C_\mathup{ol} c_\star^3 (m+1)^d \theta^{m-1}\sum^N_{i=1}\|\widetilde{g}\|^2_{\mathcal{A}{(K_i)}}\\
        &\leq \overline{C}^5 C_\mathup{ol} c_\star^3 (m+1)^d \theta^{m-1}\|\widetilde{g}\|^2_{\mathcal{A}}.
    \end{align*}
    Similarly we start from the equation \eqref{eq:Ngq},
    \begin{align*}
        \|\mathcal{N}^\mathup{glo}_i q\|_\mathcal{A}^2&\leq\|\mathcal{N}^\mathup{glo}_i q\|_\mathcal{A}^2 + \|\pi\mathcal{N}^\mathup{glo}_i q\|_s^2 =\int_{\partial K_i\cap \Gamma_N}q\mathcal{N}^\mathup{glo}_i q \di \sigma\\
        &\leq \|q\|_{L^2(\partial K_i\cap \Gamma_N)}\|\mathcal{N}^\mathup{glo}_i q\|_{L^2(\partial K_i\cap \Gamma_N)}\\
        &\leq C_\mathup{tr}\|q\|_{L^2(\partial K_i\cap \Gamma_N)}\|\mathcal{N}^\mathup{glo}_i q\|_\mathcal{A},
    \end{align*}
    which yields $\|\mathcal{N}^\mathup{glo}_i q\|_\mathcal{A}\leq C_\mathup{tr}\|q\|_{L^2(\partial K_i\cap\Gamma_N)}$.
    Then,
    \begin{align*}
        &\quad \sum^N_{i=1}\int_{\partial K_i\cap\Gamma_N}q\mathcal{N}^\mathup{glo}_i q \di \sigma
        \leq \sum^N_{i=1}\|q\|_{L^2(\partial K_i\cap \Gamma_N)}\left\|\mathcal{N}^\mathup{glo}_i q\right\|_{L^2(\Gamma_N)}\\
        &\leq \sum^N_{i=1}C_\mathup{tr}\|q\|_{L^2(\partial K_i\cap \Gamma_N)}\left\|\mathcal{N}^\mathup{glo}_i q\right\|_{\mathcal{A}} \leq \sum^N_{i=1}C_\mathup{tr}^2\|q\|_{L^2(\partial K_i\cap \Gamma_N)}^2\\
        &= C_\mathup{tr}^2\|q\|_{L^2(\Gamma_N)}^2.
    \end{align*}
\end{proof}

To further analyse function spaces $V^\mathup{glo}_\mathup{ms}$ and $V^m_\mathup{ms}$, 
we define operators 
$\mathcal{R}^\mathup{glo}\coloneqq\sum^N_{i=1}\mathcal{R}^\mathup{glo}_i$
$\colon L^2(\Omega)\rightarrow V^\mathup{glo}_\mathup{ms}$ and 
$\Rmm\colon= \sum^N_{i=1}\Rmmi\colon L^2(\Omega)\rightarrow V^m_\mathup{ms}$
where
\begin{equation}
    \mathcal{A}(\Rgloi\varphi,v)+s(\pi\Rgloi\varphi,\pi v)= s(\pi_i\varphi,\pi v),\text{ for } v\in V;
    \label{eq:Rgloi}
\end{equation}
\begin{equation}
    \mathcal{A}(\Rmmi\varphi,v)+s(\pi\Rmmi\varphi,\pi v)=s(\pi_i\varphi,\pi v),\text{ for } v\in V^m_i.
    \label{eq:Rmmi}
\end{equation}

We also remark another lemma \cite{CHUNG2018298}.
\begin{lemma}
    There exists a positive constant $C_\mathup{inv}$ such that
    for any $v\in L^2(\Omega)$, 
    there exists $\widehat{v}\in V$ with $\pi \widehat{v}=\pi v$ 
    such that $\|\widehat{v}\|_\mathcal{A}\leq C_\mathup{inv}\|\pi v\|_s.$
    \label{lemma:Cinv}
\end{lemma}
\subsubsection{Error Estimate for the Multiscale Solution}
\label{section:error1}

We now state the main result.
\begin{theorem}
    Let $\mathcal{D}^m\widetilde{g}, \mathcal{N}^m q$ and $w^m$ be the numerical solutions as defined before, 
    $w^\mathup{glo}$ defined in equation \eqref{eq:w^glo_u0},
    and $\Lambda, \theta, c_\star, c_\#, C_\mathup{tr}$ and $ C_\mathup{inv}$ be the constants defined in theorem \ref{thm:1}, lemma \ref{lemma:part1}, lemma \ref{lemma:part2},
    lemma \ref{lemma:part3}, corollary \ref{corollary:1} and lemma \ref{lemma:Cinv} respectively.
    Then
    \begin{align*}
        &\|w^m-\mathcal{D}^m\widetilde{g}+\mathcal{N}^m q +\widetilde{g}-u\|_\mathcal{A}\\
        &\leq \overline{C}\left\{\Lambda^{-1/2}\kappa_1^{-1/2}H\left(\||\bbeta|^{-1}f\|_{L^2(\Omega)}+\|\nabla u_0\|_{L^2(\Omega)}\right)\right\}\\
        &+\overline{C}^2 \sqrt{C_\mathup{ol}} c_\star^{3/2} (m+1)^{d/2}\theta^{(m-1)/2}\left\{\overline{C}^2\max(C^2_\mathup{inv},1)\|w^\mathup{glo}\|_\mathcal{B}\right.\left.+\overline{C}\|\widetilde{g}\|_\mathcal{A}
        +C_\mathup{tr}\|q\|_{L^2(\Gamma_N)}
        \right\}.
    \end{align*}
    \label{thm:2}
    Moreover, if we have $C_\mathup{inv}\theta^{(m-1)/2}(m+1)^{d/2}=O(H^2)$,
    then\[\|w^m-\mathcal{D}^m\widetilde{g}+\mathcal{N}^m q+\widetilde{g}-u\|_\mathcal{A} =O(H).\]
\end{theorem}
\begin{proof}
    By definition of $\Rgloi$, $\Rglo$ is surjective.
    Then there exists $\varphi_*\in L^2(\Omega)$ such that $w^\mathup{glo}=\Rglo\varphi_*$.
    Let $w^m_*=\Rmm\varphi_*$.
    To estimate $\Rgloi\varphi_* - \Rmmi\varphi_*$, we use the similar argument in corollary \ref{corollary:1}.
    \begin{align*}
        &\quad \|\Rgloi\varphi_* - \Rmmi\varphi_*\|_\mathcal{B}^2 \leq \overline{C}^4 C_\mathup{ol} c_\star^3 (m+1)^d \theta^{m-1}\sum^N_{i=1}
        s(\pi_i\varphi_*,\pi \Rgloi\varphi_*)\\
        &\leq\overline{C}^4 C_\mathup{ol} c_\star^3 (m+1)^d \theta^{m-1}\sum^N_{i=1}
        \|\pi_i\varphi_*\|_s\|\pi\Rglo \varphi_*\|_s \leq\overline{C}^4 C_\mathup{ol} c_\star^3 (m+1)^d \theta^{m-1}
        \sum^N_{i=1}\|\pi_i\varphi_*\|_s^2\\
        &=\overline{C}^4 C_\mathup{ol} c_\star^3 (m+1)^d \theta^{m-1}\|\pi \varphi_*\|^2_s.
    \end{align*}
    Then, by lemma \ref{lemma:Cinv},
    there exists $\widehat{\varphi}_*\in V$ such that \[\pi\widehat{\varphi}_*=\pi\varphi_*
    \text{ and }\|\widehat{\varphi}_*\|_\mathcal{A}\leq C_\mathup{inv}\|\pi \varphi_*\|_s.\]
    Putting $v=\widehat{\varphi}_*$ in equation \eqref{eq:Rgloi},
    \begin{align*}
        \|\pi\varphi_*\|_s^2=\mathcal{B}(w^\mathup{glo},\widehat{\varphi}_*)
        \leq \overline{C}\|w^\mathup{glo}\|_\mathcal{B}\|\widehat{\varphi}\|_\mathcal{B}
        \leq \overline{C}\max(C_\mathup{inv},1)\|w^\mathup{glo}\|_\mathcal{B}\|\pi {\varphi_*}\|_s.
    \end{align*}    
    So,
    \[\|\Rgloi\varphi_*-\Rmmi \varphi_*\|_\mathcal{B}^2\leq \overline{C}^6 C_\mathup{ol} c_\star^3 (m+1)^d \theta^{m-1}\max(C_\mathup{inv}^2,1)\|w^\mathup{glo}\|_\mathcal{B}^2.\]
    Now, let $u^\mathup{ms}_0\coloneqq w^m-\mathcal{D}^m\widetilde{g}+\mathcal{N}^m q$ and $u^\mathup{glo}_0\coloneqq w^\mathup{glo}-\mathcal{D}^\mathup{glo}\widetilde{g}+\mathcal{N}^\mathup{glo} q$. Then,
    \begin{align*}
        \|u-u^\mathup{ms}\|_\mathcal{A}^2&= \|u_0-u^\mathup{ms}_0\|_\mathcal{A}^2 = \mathcal{A}(u_0-u^\mathup{ms}_0,u_0-v)\\
        &\leq \overline{C}\|u_0-u^\mathup{ms}_0\|_\mathcal{A}\|u_0-v\|_\mathcal{B}.
    \end{align*}
    Putting $v=w^{m}_*-\mathcal{D}^m\widetilde{g}+\mathcal{N}^m q$,
    \begin{align*}
        &\|u_0-v\|_\mathcal{B}
        \leq \|u_0-\widetilde{u}_0\|_\mathcal{B}
        +\|\widetilde{u}_0-v\|_\mathcal{B}\\
        &\leq \|u_0-\widetilde{u}_0\|_\mathcal{A}
        + \|w^\mathup{glo}-w^m_*\|_\mathcal{B}
        + \|\mathcal{D}^\mathup{glo}\widetilde{g}-\mathcal{D}^m\widetilde{g}\|_\mathcal{B}
        + \|\mathcal{N}^\mathup{glo} q-\mathcal{N}^m q\|_\mathcal{B}.
    \end{align*}
    Altogether,
    \begin{align*}
        &\|u-u^\mathup{ms}\|_\mathcal{A}\leq \overline{C} \|u_0-v\|_\mathcal{B}\\
        &\leq \overline{C}\left\{\|u_0-\widetilde{u}_0\|_\mathcal{A}
        + \|w^\mathup{glo}-w^m_*\|_\mathcal{B}
        + \|\mathcal{D}^\mathup{glo}\widetilde{g}-\mathcal{D}^m\widetilde{g}\|_\mathcal{B}
        + \|\mathcal{N}^\mathup{glo} q-\mathcal{N}^m q\|_\mathcal{B}\right\}\\
        &\leq \overline{C}\left\{\Lambda^{-1/2}\kappa_1^{-1/2}H\left(\||\bbeta|^{-1}f\|_{L^2(\Omega)}+\|\nabla u_0\|_{L^2(\Omega)}\right)\right\}\\
        &\quad +\overline{C}^2 \sqrt{C_\mathup{ol}} c_\star^{3/2} (m+1)^{d/2}\theta^{\frac{m-1}{2}}\left\{\max(C_\mathup{inv},1)\|w^\mathup{glo}\|_\mathcal{B}\right.
        \left.+\overline{C}\|\widetilde{g}\|_\mathcal{A}
        +C_\mathup{tr}\|q\|_{L^2(\Gamma_N)}
        \right\} .
    \end{align*}
    \noindent
    Now, first recall that we have $\|u_0-\widetilde{u}_0\|_\mathcal{A}=O(1)$. Note that $\overline{C}=O(1)$.
    By the equations \eqref{eq:dgtg-dmtg} and \eqref{eq:ngq-nmq}, we have $\|\mathcal{D}^\mathup{glo}\widetilde{g}-\mathcal{D}^m\widetilde{g}\|_\mathcal{A} = O(1)$ and $\|\mathcal{N}^\mathup{glo} q-\mathcal{N}^m q\|_\mathcal{A}$ $= O(1)$.
    Assume $(m+1)^{d/2}\theta^{\frac{m-1}{2}}C_\mathup{inv} = O(H^2)$. It suffices to estimate $\|w^\mathup{glo}\|_\mathcal{A}$ and $\|\pi w^\mathup{glo}\|_s$.
    By equation \eqref{eq:w^glo_u0},
    \[\|w^\mathup{glo}\|_\mathcal{A}^2=\mathcal{A}(u_0,w^\mathup{glo})+\mathcal{A}(\mathcal{D}^\mathup{glo}\widetilde{g},w^\mathup{glo})-\mathcal{A}(\mathcal{N}^\mathup{glo} q,w^\mathup{glo})\]
    giving us
    \[ \|w^\mathup{glo}\|_\mathcal{A}\leq \|u_0\|_\mathcal{A}+ \|\mathcal{D}^\mathup{glo}\widetilde{g}\|_\mathcal{A} + \|\mathcal{N}^\mathup{glo} q\|_\mathcal{A} =O(1).\]
    On the other hand,
    \[\|\pi w^\mathup{glo}\|_s \leq \|\pi u_0\|_s + \|\pi\mathcal{D}^\mathup{glo}\widetilde{g}\|_s + \|\pi\mathcal{N}^\mathup{glo} q\|_s \leq O(H^{-1})+ O(1).\]
    \noindent
    Therefore, we can obtain 
    $\|u^\mathup{ms}-u\|_\mathcal{A} = O(H)$.
\end{proof}

\subsection{Numerical Experiments}
\label{section:numerical_experiments_1}

In this section, we will demonstrate the method via several numerical examples in a high-contrast setting and verify the significance of the inflow condition. 
For simplicity, we take point-wise isotropic coefficients, $\boldsymbol{A}(x)=\kappa(x)\boldsymbol{I}$, the domain $\Omega = [0,1]\times[0,1]$. We will calculate the reference solutions on a $200\times200$ mesh with the bilinear Lagrange finite element method. The medium $\kappa$ is presented in Figure \ref{fig:mediumA} and the source term in Figure \ref{fig:source}. The experiments are each tested for coarse mesh $H = \frac{1}{10}, \frac{1}{20}$ and $\frac{1}{40}$ with the fixed number $l_m$ of eigenfunctions to generate the auxiliary space $V^\mathup{aux}$. By our experiments, we tested $l_m=3$ to be sufficient to verify our results. 
 \begin{figure}[ht]
\begin{subfigure}[b]{0.45\textwidth}
    \centering
    \includegraphics[width=\textwidth]{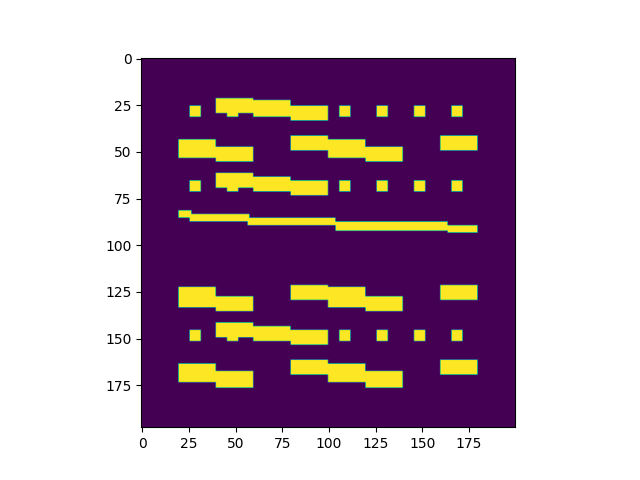}
    \caption{}
    \label{fig:mediumA}
\end{subfigure}
\hfill
\begin{subfigure}[b]{0.45\textwidth}
    \centering
    \includegraphics[width=\textwidth]{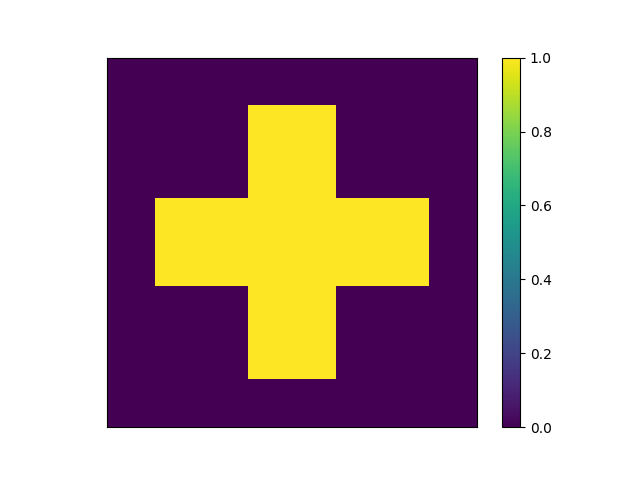}
    \caption{}
    \label{fig:source}
\end{subfigure}
\caption{\ref{fig:mediumA} Medium $\kappa$ \ref{fig:source} The source term $f$}
\end{figure}

\subsubsection{Example 1}
We first look at the Dirichlet condition, via considering the following problem:
\begin{equation}
    \begin{cases}
        -\nabla\cdot\left(\kappa(x_1,x_2) \nabla u\right)+\bbeta(x_1,x_2)\cdot\nabla u=f &\text{ for } (x_1,x_2)\in\Omega\\
        u(x_1,x_2)=\widetilde{g}(x_1,x_2)=x_1^2 + e^{x_1 x_2}&\text{ for }(x_1,x_2)\in \{0,1\}\times[0,1]
        \end{cases}
\end{equation}
where
$$\bbeta(x_1,x_2)=\left[\cos(18\pi x_2)\sin(18\pi x_1),
-\cos(18\pi x_1)\sin(18\pi x_2)\right]^T.$$

\begin{table}[ht]
\centering
\begin{tabular}{c|cccccc}
\hline\hline
$\kappa_1/\kappa_0$ & \multicolumn{1}{c}{$10^1$} & \multicolumn{1}{c}{$10^2$} & \multicolumn{1}{c}{$10^3$} & \multicolumn{1}{c}{$10^4$} & $10^5$ & $10^6$ \\\hline
$E^3_a$ & 9.10E-04 & 7.17E-04 & 3.64E-04 & 1.37E-04 & 7.16E-05 & 6.10E-05 \\
$E^3_L$ & 5.53E-06 & 5.64E-06 & 5.66E-06 & 5.66E-06 & 5.84E-06 & 1.08E-05 \\\hline
$D^3_a$ & 2.29E-06 & 1.06E-06 & 9.14E-07 & 8.99E-07 & 8.97E-07 & 8.97E-07 \\
$D^3_L$ & 4.24E-06 & 3.76E-06 & 3.77E-06 & 3.78E-06 & 3.78E-06 & 3.78E-06\\\hline\hline
\end{tabular}
\caption{Example 1: Dirichlet Boundary Condition with varying contrast levels $\kappa_1/\kappa_0$ and fixed $l_m=3$, $N_{ov}=4$, $H=1/40$}
\label{tab1.1}
\end{table}

\begin{table}[ht]
\centering
\begin{tabular}{c|ccccc}
\hline\hline
$l_m$ & \multicolumn{1}{c}{1} & \multicolumn{1}{c}{2} & \multicolumn{1}{c}{3} & \multicolumn{1}{c}{4} & 5 \\\hline
$E^3_a$ & 8.75E-03 & 2.88E-03 & 1.54E-03 & 1.78E-03 & 9.11E-04 \\
$E^3_L$ & 2.82E-04 & 3.03E-05 & 1.11E-05 & 1.38E-05 & 5.00E-06 \\\hline
$D^3_a$ & 2.15E-04 & 5.67E-05 & 2.41E-05 & 2.79E-05 & 1.55E-05 \\
$D^3_L$ & 1.33E-03 & 3.51E-04 & 9.47E-05 & 1.18E-04 & 6.38E-05\\\hline\hline
\end{tabular}
\caption{Example 1: Dirichlet Boundary Condition with varying numbers $l_m$ of eigenfunctions and fixed $\kappa_1/\kappa_0=10^3$, $N_{ov}=3$, $H=1/40$}
\label{tab1.2}
\end{table}

To simplify notations, we denote the relative errors for the Dirichlet corrector
\[D^m_a \coloneqq\frac{\|\mathcal{D}^m\widetilde{g}-\mathcal{D}^\mathup{glo}\widetilde{g}\|_\mathcal{A}}{\|\mathcal{D}^\mathup{glo}\widetilde{g}\|_\mathcal{A}}
\text{ and }
D^m_L \coloneqq \frac{\|\mathcal{D}^m\widetilde{g}-\mathcal{D}^\mathup{glo}\widetilde{g}\|_{L^2(\Omega)}}{\|\mathcal{D}^\mathup{glo}\widetilde{g}\|_{L^2}};\]

For the Neumann condition corrector,
\[N^m_a \coloneqq\frac{\|\mathcal{N}^m q-\mathcal{N}^\mathup{glo} q\|_\mathcal{A}}{\|\mathcal{N}^\mathup{glo} q\|_\mathcal{A}}
\text{ and }
N^m_L \coloneqq \frac{\|\mathcal{N}^m q-\mathcal{N}^\mathup{glo} q\|_{L^2(\Omega)}}{\|\mathcal{N}^\mathup{glo} q\|_{L^2}};\]
to measure errors and $\Lambda^\prime = \max_i\lambda^{l_m}_i$.
As for the error estimate,
\[ E^m_a \coloneqq \frac{\|u^\mathup{ms}-u\|_\mathcal{A}}{\|u\|_\mathcal{A}} \text{ and }
E^m_L \coloneqq \frac{\|u^\mathup{ms}-u\|_{L^2}}{\|u\|_{L^2}}.\]

\begin{table}[ht]
\centering
\begin{tabular}{@{\extracolsep{\fill}}cccc}
\hline\hline
 & \multicolumn{1}{c}{1/10} & \multicolumn{1}{c}{1/20} & \multicolumn{1}{c}{1/40} \\\hline
    \multicolumn{1}{l}{$\Lambda$} & 2.273418 & 2.328070 & 3.185349 \\\hline
    $E^3_a$ & \textbf{3.63E-03} & 1.25E-03 & 1;54E-03 \\
    $E^4_a$ & 3.63E-03 & \textbf{1.16E-03(32.0\%)} & 3.64E-04 \\
    $E^5_a$ & 3.63E-03 & 1.16E-03 & \textbf{3.60E-04(31.0\%)} \\\hline
    $E^3_L$ & \textbf{1.49E-04} & 3.10E-05 & 1.10E-05 \\
    $E^4_L$ & 1.49E-04 & \textbf{3.10E-05(20.8\%)} & 6.00E-06 \\
    $E^5_L$ & 1.49E-04 & 3.10E-05 & \textbf{6.00E-06(19.4\%)}\\ \hline\hline
\end{tabular}
\caption{Example 1 with $\kappa_1/\kappa_0=10^4$, $l_m=3$}
\label{tab:Diri_solution}
\end{table}

As can be seen in Table \ref{tab1.1}, the error of the Dirichlet corrector decays exponentially as we increase the number of oversampling layers. From Table \ref{tab1.2}, the number of eigenfunctions does improve the results for both the overall multiscale solutions and the boundary correctors. Most importantly, from Table \ref{tab:Diri_solution}, we can see a second-order convergence when increasing the oversampling layers and decreasing the coarse mesh. This also echoes the idea that the convergence of $u^\mathup{ms}$ depends on the oversampling layers $N_{ov}$.

\subsubsection{Example 2}

Another demonstration comes from the Neumann and Robin conditions. In this problem, we first consider the same velocity field $\bbeta$:
\begin{equation}
    \begin{cases}
        -\nabla\cdot\left(\kappa(x_1,x_2) \nabla u\right)+\bbeta(x_1,x_2)\cdot\nabla u=f &\text{ for } (x_1,x_2)\in\Omega\\
        u(x_1,x_2)=\widetilde{g}(x_1,x_2)=x_1^2 + e^{x_1 x_2}&\text{ for }(x_1,x_2)\in \{0,1\}\times[0,1]\\
        bu + \bnu\cdot(A\nabla u -\bbeta u) = -1 &\text{ for } x_1 = 0 \text{ and } x_2 \in [0,1]\\
        bu + \bnu\cdot(A\nabla u -\bbeta u) = 1 &\text{ for } x_1 = 1 \text{ and } x_2 \in [0,1]\\
        bu + \bnu\cdot(A\nabla u -\bbeta u) = 1 &\text{ for } x_1 = (0.5,1] \text{ and } x_2 =0\\
        bu + \bnu\cdot(A\nabla u -\bbeta u) = 0 &\text{ for } x_1 = [0, 0.5] \text{ and } x_2 = 0
        \end{cases}
\end{equation}
where $b(x_1,x_2)= \kappa(x_1,x_2)$ is the Robin coefficient.

For the following results, the numbers in brackets show the relative error for varying numbers of oversampling layers.
 \begin{figure}[ht]
 \begin{subfigure}[b]{0.33\textwidth}
    \centering
    \includegraphics[width=\textwidth]{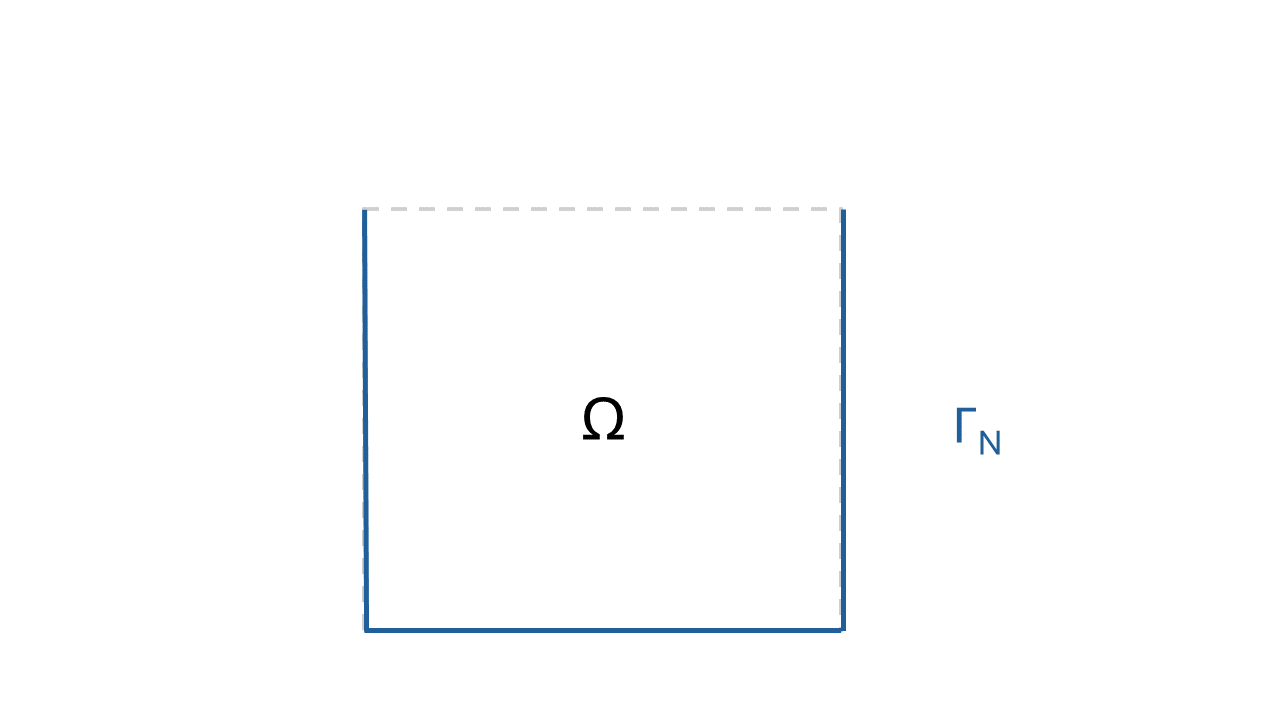}
    \caption{$\Gamma_{N}$}
    \label{fig:Neum_bound}
\end{subfigure}
\begin{subfigure}[b]{0.33\textwidth}
    \centering
    \includegraphics[width=\textwidth]{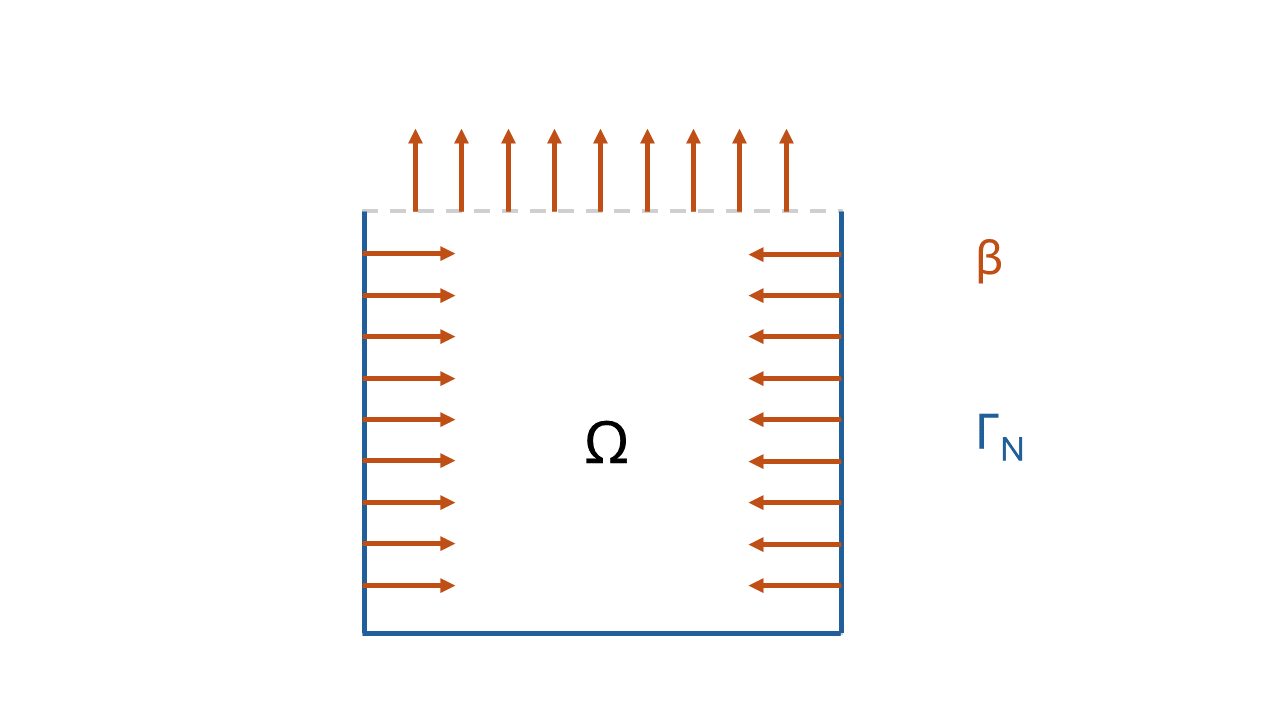}
    \caption{$\bbeta_\mathup{in}$}
    \label{fig:Neum_bound_inflow}
\end{subfigure}
\begin{subfigure}[b]{0.33\textwidth}
    \centering
    \includegraphics[width=\textwidth]{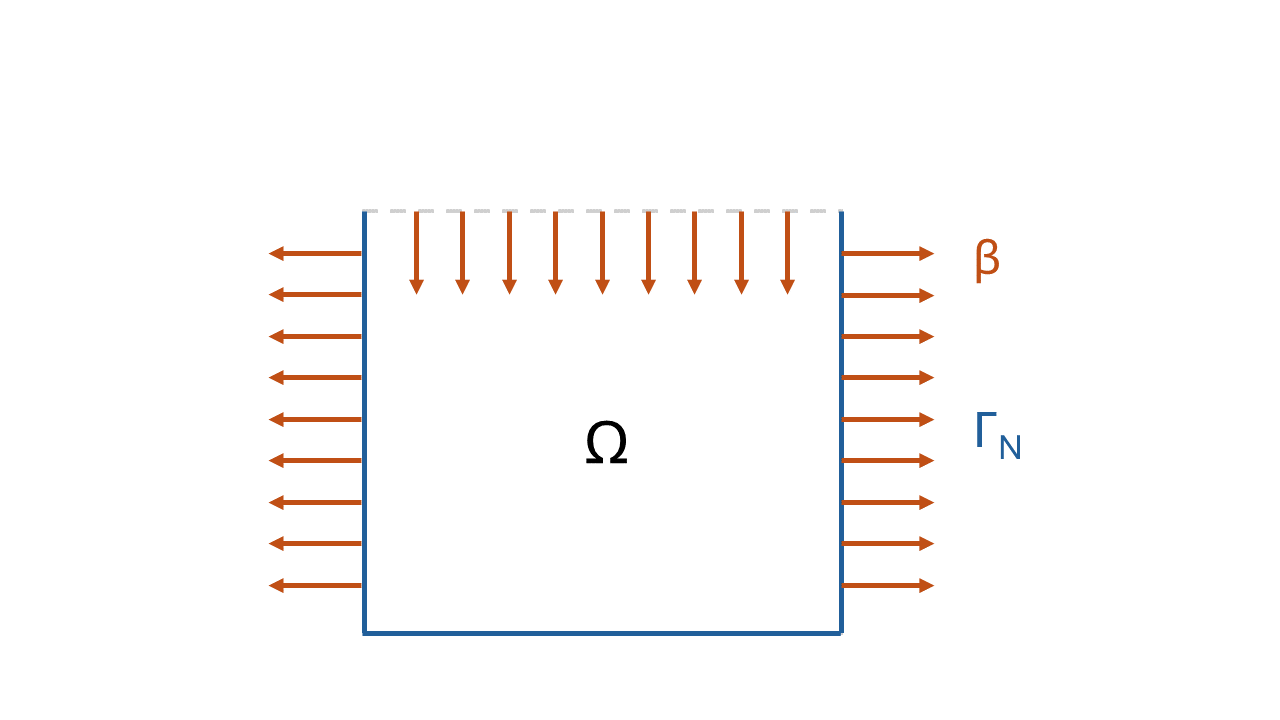}
    \caption{$\bbeta_\mathup{out}$}
    \label{fig:Neum_bound_outflow}
\end{subfigure}
\caption{(\ref{fig:Neum_bound}) Example 2 
 (\ref{fig:Neum_bound_inflow}) Example 3 (\ref{fig:Neum_bound_outflow}) Example 4}
\end{figure}

\begin{table}[ht]
\centering
\begin{tabular}{@{\extracolsep{\fill}}cccc}
\hline\hline
 & \multicolumn{1}{c}{1/10} & \multicolumn{1}{c}{1/20} & \multicolumn{1}{c}{1/40} \\\hline
    \multicolumn{1}{l}{$\Lambda$} & 2.273414 & 2.328069 & 3.185349 \\\hline
    $E^3_a$ & \textbf{6.22E-03} & 3.53E-03 & 6.08E-03 \\
    $E^4_a$ & 6.21E-03 & \textbf{3.13E-03(50.3\%)} & 1.64E-03 \\
    $E^5_a$ & 6.21E-03 & 3.13E-03 & \textbf{1.62E-03(51.8\%)} \\\hline
    $E^3_L$ & \textbf{5.52E-04} & 1.60E-04 & 1.28E-04 \\
    $E^4_L$ & 5.52E-04 & \textbf{1.60E-04(29.0\%)} & 6.20E-05 \\
    $E^5_L$ & 5.52E-04 & 1.60E-04 & \textbf{6.20E-05(9.92\%)}\\ \hline\hline
\end{tabular}
\caption{Example 2 with $\kappa_1/\kappa_0=10^4$, $l_m=3$}
\label{tab1.3}
\end{table}

From Table \ref{tab1.3}, we also can observe the second-order convergence with respect to $H$. Moreover, we compare the effect of the inflow conditions using the following two examples. 

\subsubsection{Example 3}

In this model, we consider the inflow condition on $\Gamma_N$.
\[\bbeta_\mathup{in}(x_1,x_2) = \bbeta + c_\mathup{flow}\left[\frac{1}{2}-x_1,x_1\right]^T.\]
Notice that $\bbeta_\mathup{in}\cdot\bnu\leq 0$ on $\Gamma_N$ as shown in Figure \ref{fig:Neum_bound_inflow}. The constant $c_\mathup{flow}$ is proportional to the magnitude of the velocity on the boundary $\Gamma_N$,

\begin{table}[ht]
\centering
\begin{tabular}{c|ccc}
\hline\hline
H       & 1/10     & 1/20     & 1/40     \\
\hline
$N^2_a$ & \textbf{1.15E-03} & 1.02E-03 & 1.63E-03 \\
$N^3_a$ & {8.15E-05} & 5.62E-05 & 1.33E-04 \\
$N^4_a$ & 7.35E-06 & \textbf{4.79E-06(41.6\%)}& 1.21E-05 \\
$N^5_a$ & 6.18E-07 & 4.29E-07 & {1.09E-06} \\
$N^6_a$ & 6.57E-08 & 1.75E-07 & \textbf{1.26E-07(2.63\%)}\\
\hline
$N^2_L$ & \textbf{9.08E-04}& 7.69E-04 & 8.97E-04 \\
$N^3_L$ & {5.71E-05}& 3.30E-05 & 7.55E-05 \\
$N^4_L$ & 4.95E-06 & \textbf{2.85E-06(0.31\%)}& 7.63E-06 \\
$N^5_L$ & 4.62E-07 & 2.77E-07 & {6.37E-07}\\
$N^6_L$ & 4.25E-08 & 1.08E-07 & \textbf{6.83E-08(3.40\%)}\\
\hline\hline
\end{tabular}
\caption{Example 3: $\mathcal{N}^m q$ error with $l_m=3$ and $\kappa_1/\kappa_0=10^3$ and $c_\mathup{flow}=2$}
\label{tab:3.5}
\end{table}

\begin{table}[ht]
\centering
\begin{tabular}{c|ccc}
\hline\hline
H       & 1/10     & 1/20     & 1/40     \\
\hline
$E^2_a$ & \textbf{1.90E-01} & 3.73E-01 & 5.59E-01 \\
$E^3_a$ & {1.13E-02} & 2.29E-02 & 1.88E-01 \\
$E^4_a$ & 8.77E-03 & \textbf{5.06E-03(2.67\%)}& 7.85E-03 \\
$E^5_a$ & 8.76E-03 & 5.01E-03 & {1.66E-03 }\\
$E^6_a$ & 8.76E-03 & 5.02E-03 & \textbf{1.60E-03(31.6\%)}\\
\hline
$E^2_L$ & \textbf{1.60E-01} & 3.68E-01 & 6.73E-01 \\
$E^3_L$ & {1.04E-03} & 3.48E-03 & 1.23E-01 \\
$E^4_L$ & 8.87E-04 & \textbf{3.75E-04(0.23\%)}& 2.55E-04 \\
$E^5_L$ & 8.88E-04 & 3.76E-04 & {6.83E-05} \\
$E^6_L$ & 8.88E-04 & 3.76E-04 & \textbf{6.84E-05(18.2\%)}\\
\hline\hline
\end{tabular}
\caption{Example 3: solution error with $l_m=3$ and $\kappa_1/\kappa_0=10^3$ and $c_\mathup{flow}=2$}
\label{tab:3.6}
\end{table}

As can be seen in Tables \ref{tab:3.5} and \ref{tab:3.6}, similar observations of the results are shown, resembling our theoretical analysis and the numerical results in the examples. 

\subsubsection{Example 4}

We consider the same setting as Example 3, but with the following velocity field 
$\bbeta_\mathup{out} = -\bbeta_\mathup{in}$
with $c_\mathup{flow}>0$.
Note that $\nabla\cdot\bbeta_\mathup{out}=0$ on $\Omega$ but $\bbeta_\mathup{out}\cdot\bnu > 0 $ on $x_1=0$ and $x_1=1$, shown in Figure \ref{fig:Neum_bound_outflow}.

\begin{table}[ht]
\centering
\begin{tabular}{c|ccc}
\hline\hline
H & \multicolumn{1}{c}{1/10} & \multicolumn{1}{c}{1/20} & \multicolumn{1}{c}{1/40} \\\hline
$\Lambda$&0.53025 &0.97397 & 3.25683\\\hline
$E^3_a$ & \textbf{1.30E-02} & 3.45E-02 & 1.48E-01 \\
$E^4_a$ & 4.96E-03 & \textbf{3.65E-03(28.1\%)}& 1.15E-02 \\
$E^5_a$ & 4.87E-03 & 2.57E-03 & \textbf{1.27E-03(34.8\%)}\\\hline
$E^3_L$ & \textbf{6.34E-04} & 4.27E-03 & 7.46E-02 \\
$E^4_L$ & 2.28E-04 & \textbf{1.01E-04(15.9\%)}& 5.20E-04 \\
$E^5_L$ & 2.29E-04 & 8.60E-05 & \textbf{1.37E-05(13.6\%)}\\\hline\hline
\end{tabular}
\caption{Example 4 with $l_m=3$, $c_\mathup{flow}=3$ and  $\kappa_1/\kappa_0=10^3$}
\label{tab:outflow}
\end{table}

\begin{figure}[ht]
    \centering
    \includegraphics[scale=0.6]{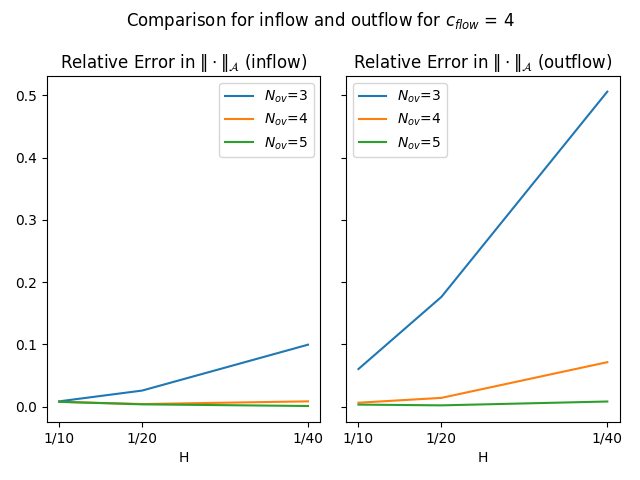}
    \caption{Comparison of inflow and outflow case in $L^2$ norm: $l_m=3$ and $\kappa_1/\kappa_0=10^3$}
    \label{fig:in_outflow_A4}
\end{figure}
\begin{figure}[ht]
    \centering
    \includegraphics[scale=0.6]{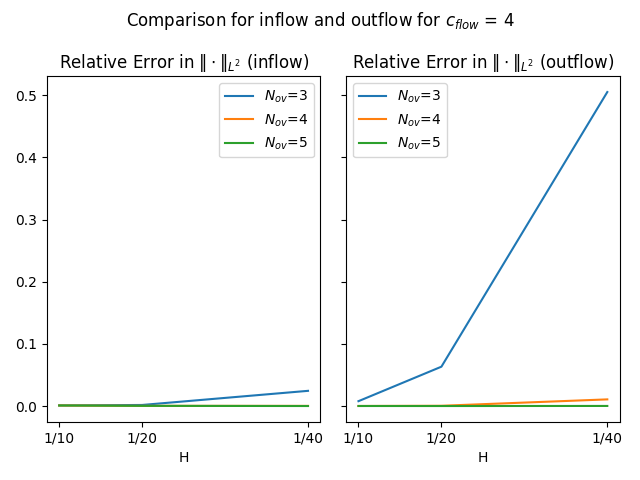}
    \caption{Comparison of inflow and outflow case in $\mathcal{A}$ norm: $l_m=3$ and $\kappa_1/\kappa_0=10^3$}
    \label{fig:in_outflow_L4}
\end{figure}

In Table \ref{tab:outflow}, second-order convergence in $H$ is still observed when given a large enough number of oversampling layers and a fixed $c_\mathup{flow}=3$. However, as we increase the velocity field $c_\mathup{flow}=4$, we see that the problem is more demanding. Without a larger number of oversampling layers, reducing the coarse mesh $H$ alone may not be enough to improve the approximation. However, when given a sufficient number of layers, the job can still be done. The comparison is more apparent in Figure \ref{fig:in_outflow_A4} and Figure \ref{fig:in_outflow_L4}, where the outflow case converges much slower in both $\|\cdot\|_{L^2}$ and $\|\cdot\|_\mathcal{A}$ compared to the inflow case.

\section{Time Dependent Convection-Diffusion Initial Boundary Value Problems}
\label{chapter:2}
\subsection{Derivation of the Method}
We present the convection diffusion initial boundary value problem as follows:
\begin{equation}
\begin{cases}
    \partial_t u(x,t) -\nabla \cdot(\boldsymbol{A}(x) \nabla u(x,t))+ \bbeta(x) \cdot\nabla u(x,t) = f(x,t) &\text{ on }\Omega\times (0,T]\\
    u(x,t) = g(x,t) &\text{ on }\Gamma_D\times (0,T]\\
    b(x) u(x,t) + \bnu \cdot (\boldsymbol{A}(x) \nabla u(x.t)- \bbeta(x) u(x,t))= q(x,t) &\text{ on }\Gamma_N\times (0,T]\\
    u(x,0) = u_\mathup{init}(x) &\text{ on } \Omega.
\end{cases}
\label{eqn:unsteadyBVP}
\end{equation}
Notice $f$ is still independent of the solution $u$.
The variational formulation becomes: 
find $u_0(\cdot,t)\in V$ such that for $v\in V$,
\begin{align}
    (\partial_t u_0,v) + \mathcal{A}(u_0,v) &= (f,v) + (q,v)_{\Gamma_N} - \mathcal{A}(\widetilde{g},v) - (\widetilde{g}_t,v)\label{eqn:unsteadyweakBVP}\\
    (u_0(\cdot, 0),v) &= (u_\mathup{init}(x)-\widetilde{g}(\cdot,0),v)\label{eqn:unsteadyInitial}
\end{align}
and the actual solution would be $u(\cdot,t) = u_{0}(\cdot,t)+\widetilde{g}(\cdot,t)$.
We will use the same auxiliary spaces $V^\mathup{aux}$ and multiscale space $V^m_\mathup{ms}$ as in the previous section, and thereby the same set of correctors $\mathcal{D}^m \widetilde{g}$ and $\mathcal{N}^m q$. 
In particular, for $t\in (0,T]$,
\begin{enumerate}
    \item Find $\mathcal{D}^m_i \widetilde{g}(\cdot, t)\in V^m_i$ such that $v\in V^m_i$,
    \begin{equation}
    \begin{cases}
        (\mathcal{D}^m_i \widetilde{g}_t, v) + \mathcal{B}(\mathcal{D}^m_i \widetilde{g},v) = (\widetilde{g}_t, v)_{(K_i)} + \mathcal{A}_{(K_i)}(\widetilde{g},v)\\
        \mathcal{B}(\mathcal{D}^m_i \widetilde{g}(\cdot, 0),v) = \mathcal{A}_{(K_i)}(\widetilde{g}(\cdot,0),v).
    \end{cases}
    \label{eq:dg_time_variant_ms}
    \end{equation}
    Denote $\mathcal{D}^m \widetilde{g} = \sum_{i=1}^{N}\mathcal{D}^m_i \widetilde{g}$.
    \item Find $\mathcal{N}^m_i q(\cdot, t) \in V^m_i$ such that $v\in V^m_i$,
    \begin{equation}
    \begin{cases}
        (\mathcal{N}^m_i q_t, v) + \mathcal{B}(\mathcal{N}^m_i q,v)= \int_{\Gamma_N\cap \partial K_i}qv \di \sigma\\
        \mathcal{B}(\mathcal{N}^m_i q(\cdot, 0),v) = \int_{\Gamma_N\cap\partial K_i }q(\cdot, 0)v \di \sigma.
    \end{cases}
    \label{eq:nq_time_variant_ms}
    \end{equation}
    Denote $\mathcal{N}^m q = \sum^N_{i=1} \mathcal{N}^m_i q.$
    
\item find $w^m(\cdot, t) \in V^m_\mathup{ms}$ such that for $v\in $$V^m_\mathup{ms}$
\begin{equation}
\begin{aligned}
(w^m_t,v) + \mathcal{A}(w^m,v) &= (f,v) + (q,v)_{\Gamma_N} -\mathcal{A}(\widetilde{g},v)-(\widetilde{g}_t,v) \\
& + (\mathcal{D}^m \widetilde{g}_t, v) + \mathcal{A}(\mathcal{D}^m \widetilde{g}, v) -(\mathcal{N}^m q_t,v) -\mathcal{A}(\mathcal{N}^m q, v) ,
\end{aligned}
\label{eqn:ms_t_weakBVP}
\end{equation}
\end{enumerate}
\begin{equation}
    (w^m(\cdot, 0), v) = \left(u_\mathup{init}-\widetilde{g}(\cdot,0)+\mathcal{D}^m \widetilde{g}(\cdot,0)-\mathcal{N}^m q(\cdot,0),v\right) .
\end{equation} 

Then the multiscale approximation becomes 
\begin{equation}
u^\mathup{ms} = u^\mathup{ms}_0 + \widetilde{g} = w^m- \mathcal{D}^m \widetilde{g} + \mathcal{N}^m q+\widetilde{g}.
    \label{eq:time_dpdt_solution}
\end{equation}
\subsection{Analysis}

To test the performance of the oversampling layers, we define $\mathcal{D}^\mathup{glo}\widetilde{g}=\sum^N_{i=1}\mathcal{D}^\mathup{glo}_i\widetilde{g}$ and $\mathcal{N}^\mathup{glo} q=\sum^N_{i=1}\mathcal{N}^\mathup{glo}_i q$ where $\mathcal{D}^\mathup{glo}_i\widetilde{g},\mathcal{N}^\mathup{glo}_i q\in V^\mathup{glo}_\mathup{ms}$ satisfies that for all $v\in V$,
\begin{equation}
    \begin{cases}
        (\mathcal{D}^\mathup{glo}_i \widetilde{g}_t, v) + \mathcal{B}(\mathcal{D}^\mathup{glo}_i \widetilde{g},v) = (\widetilde{g}_t, v)_{(K_i)} + \mathcal{A}_{(K_i)}(\widetilde{g},v)\\
        \mathcal{B}(\mathcal{D}^\mathup{glo}_i \widetilde{g}(\cdot, 0),v) = \mathcal{A}_{(K_i)}(\widetilde{g}(\cdot,0),v),
    \end{cases}
    \label{eq:dg_time_variant}
    \end{equation},
    and 
    \begin{equation}
    \begin{cases}
        (\mathcal{N}^\mathup{glo}_i q_t, v) + \mathcal{B}(\mathcal{N}^\mathup{glo}_i q,v)= \int_{\Gamma_N\cap \partial K_i}qv \di \sigma\\
        \mathcal{B}(\mathcal{N}^\mathup{glo}_i q(\cdot, 0),v) = \int_{\Gamma_N\cap\partial K_i }q(\cdot, 0)v \di \sigma.
    \end{cases}
    \label{eq:nq_time_variant}
    \end{equation}

We will give an overview of the analysis.
Define for $v\in V$,
$\|v\|_\mathcal{E}^2 = \|v(\cdot,T)\|_{L^2}^2 + \int^T_0 \|v\|_\mathcal{B}^2$.
Note that since $\|\cdot\|_\mathcal{A}$ is a quasi-norm, so are $\|\cdot\|_\mathcal{B}$ and $\|\cdot\|_\mathcal{E}$. 
Suppose $\widetilde{u}(\cdot,t)$ is the elliptic projection of the solution $u$, i.e.
\[\mathcal{A}(u-\widetilde{u},v) =0 \text{ for } v\in V^\mathup{glo}_\mathup{ms}.\]
Akin to the previous treatment, our strategy is to decompose $u-v$ into two parts:
\[
u-u^\mathup{ms} = (u-\widetilde{u})+ (\widetilde{u}-u^\mathup{ms}).
\]
The error of the former term is computable while that of the latter term can be bounded by a specific choice $v\in V^m_\mathup{ms}$.
In particular, for any $v\in V^m_\mathup{ms}$,
\begin{align*}
    &\int^T_0\left[((u-u^\mathup{ms})_t,u-u^\mathup{ms}) + \mathcal{A}(u-u^\mathup{ms},u-u^\mathup{ms})\right]\\
    &=\int^T_0\left[((u-u^\mathup{ms})_t,u-v) + \mathcal{A}(u-u^\mathup{ms},u-v)\right]\\
     &\leq \left.(u-u^\mathup{ms},u-v)\right|^T_0
     -\int^T_0 ((u-v)_t,u-u^\mathup{ms}) + \int^T_0 \mathcal{A}(u-u^\mathup{ms},u-v)\\
     &\leq \|(u-u^\mathup{ms})(\cdot,T)\|_{L^2}\|(u-v)(\cdot,T)\|_{L^2}
     +\|(u-u^\mathup{ms})(\cdot,0)\|_{L^2}\|(u-v)(\cdot,0)\|_{L^2}\\
     &+\sqrt{\int^T_0\|(u-v)_t\|_{L^2}^2}\sqrt{\int^T_0\|u-u^\mathup{ms}\|_{L^2}^2}
     +\frac{3}{4}{\int^T_0\|u-u^\mathup{ms}\|_\mathcal{A}^2}
     +\frac{1}{3}\overline{C}^2{\int^T_0\|u-v\|_\mathcal{B}^2}.
\end{align*}

Hence, by the repeated use of Young's inequality and the Cauchy Schwartz inequality, 
\begin{equation}
\begin{aligned}
    \|(u-u^\mathup{ms})(\cdot,T)\|_{L^2}^2+\int^T_0\|u-u^\mathup{ms}\|_\mathcal{A}^2
    &\leq 4\|(u-u^\mathup{ms})(\cdot,0)\|_{L^2}^2 
     + 2\|(u-v)(\cdot,0)\|_{L^2}^2+ 4\overline{C}^2\|u-v\|_{\mathcal{E}}^2\\
    &+4\sqrt{\int^T_0\|(u-v)_t\|_{L^2}^2}\sqrt{\int^T_0\|u-u^\mathup{ms}\|_{L^2}^2}\\
    &=: (i) + (ii) + (iii) + (iv).
\end{aligned}
\label{eq:main_t}
\end{equation}

Note that $(i)$ has been investigated in the time-independent case. 
Now, define $w^\mathup{glo}\in V^\mathup{glo}_\mathup{ms}$ such that for any $v\in V$,
\begin{equation}
\mathcal{A}(w^\mathup{glo},v) = \mathcal{A}(\widetilde{u} - \mathcal{D}^\mathup{glo}\widetilde{g} + \mathcal{N}^\mathup{glo} q,v). 
\label{eq:elliptic_projection_t}
\end{equation}

Then, with the same set of $\Rglo\coloneqq\sum^N_{i=1}\Rgloi$ and $\Rmm\coloneqq\sum^N_{i=1}\Rmmi$ operators in the last section, there is $\varphi(\cdot,t) \in L^2(\Omega)$ such that
\begin{align*}
    (\Rgloi\varphi_t,v)+\mathcal{B}(\Rgloi\varphi,v)=s(\pi_i \varphi, \pi v)&\text{ for }v\in V,\\
    (\Rmmi\varphi_t,v)+\mathcal{B}(\Rmmi\varphi,v)= s(\pi_i \varphi,\pi v)&\text{ for }v\in V^m_\mathup{ms}.
\end{align*}
with initial conditions $\mathcal{B}(\Rgloi\varphi(\cdot,0),v)=s(\pi_i\varphi(\cdot,0),\pi v)$ for $v\in V$ and $\mathcal{B}(\Rmmi\varphi(\cdot,0),v)=s(\pi_i\varphi(\cdot,0),\pi v)$ for $v\in V^m_\mathup{ms}$. 
Then, by the surjectivity of $\Rglo$, we can find $\varphi_*(\cdot,t)\in L^2(\Omega)$ such that 
\[
\Rglo\varphi_* = w^\mathup{glo} \text{ and define } w^m_*\coloneqq \Rmm\varphi_*.
\]

By putting $v=w^m_* - \mathcal{D}^m \widetilde{g} + \mathcal{N}^m q$, we can further decompose 
\[
 u - v = (u-\widetilde{u}) + (w^\mathup{glo}-w^m_*)+ (\mathcal{D}^\mathup{glo}\widetilde{g}-\mathcal{D}^m \widetilde{g}) + (\mathcal{N}^\mathup{glo} q-\mathcal{N}^m q).
\]

The error analysis of these terms will suffice that of the remaining terms $(ii), (iii)$, and $(iv)$ in equation \eqref{eq:main_t}. 
$u-\widetilde{u}$ is first dealt with
and the rest in another abstract problem.
\subsubsection{Elliptic Projection}
\label{section:elliptic_t}

We first give an error bound for the elliptic projection.

\begin{lemma}
Let $\widetilde{u}_0\in V^\mathup{glo}_\mathup{ms}$ be the elliptic projection of $u_0$ onto $V^\mathup{glo}_\mathup{ms}$, i.e.
\[\mathcal{A}(\widetilde{u}_0,v) = \mathcal{A}(u_0,v) \text{ for }v\in V^\mathup{glo}_\mathup{ms}.\]
Let $\widetilde{u} = \widetilde{u}_0 + \widetilde{g}$. Then,
    \begin{equation}
    \|u-\widetilde{u}\|_\mathcal{A}\leq \Lambda^{-1/2}\|\widetilde{\kappa}^{-1/2}(f-u_t)\|_{L^2}.    
    \end{equation}
    \label{lemma:bound_t_glo}
\end{lemma}
\begin{proof}
    By the definition of $\widetilde{u}$, 
    since $\mathcal{A}(\widetilde{u}-u, v) = 0$ for $v\in V^\mathup{glo}_\mathup{ms}$,
    $\pi (\widetilde{u}-u) = 0$.
    Then $\mathcal{A}(\widetilde{u}_0, \widetilde{u}-u) = 0$.
    This grants us
    \begin{align*}
        \|\widetilde{u}-{u}\|_\mathcal{A}^2 
        &= \|\widetilde{u}_0-u_0\|_\mathcal{A}^2\\
        &= \mathcal{A}(u_0, \widetilde{u} - u)\\
        &= (f,\widetilde{u}-u) + ((u_0)_t,\widetilde{u}-u)\\
        &= (f-u_t, \widetilde{u}-u)\\
        &\leq \Lambda^{-1/2}\|\widetilde{\kappa}^{-1/2}(f-u_t)\|\|\widetilde{u}-u\|_\mathcal{A}.
    \end{align*}
\end{proof}

We also can obtain the $L^2$-norm of the global estimate $\widetilde{u}$.
\begin{lemma} With the notations in lemma \ref{lemma:bound_t_glo}, we have
    \[\|u-\widetilde{u}\|_{L^2}\leq \overline{C} Hc_\#\Lambda^{-1/2}\|\widetilde{\kappa}^{-1/2}(f-u_t)\|_{L^2}.\]
\end{lemma}
\begin{proof}
    Let $z\in V$ and $\widetilde{z}\in V^\mathup{glo}_\mathup{ms}$ such that \begin{align*}
    \mathcal{A}(z,v) = (u-\widetilde{u},v) 
    &\text{ for }v\in V \\
    \mathcal{A}(\widetilde{z},v) = (u-\widetilde{u},v) &\text{ for } v \in V^\mathup{glo}_\mathup{ms}.
    \end{align*}
    Note that $\mathcal{A}(z-\widetilde{z},v) = 0$ for $v \in V^\mathup{glo}_\mathup{ms}$. This grants us $\mathcal{A}(\widetilde{z},z-\widetilde{z})=0$ and $\mathcal{A}(\widetilde{z}, u-\widetilde{u})=0$.
    Then, 
    \begin{align*}
        \|z-\widetilde{z}\|_\mathcal{A}^2 &= \mathcal{A}(z-\widetilde{z}, z-\widetilde{z})\\
        &= \mathcal{A}(z,z-\widetilde{z})\\
        &= (u-\widetilde{u}, z-\widetilde{z})\\
        &\leq \|u-\widetilde{u}\|_{L^2}\|z-\widetilde{z}\|_{L^2}\\
        &\leq Hc_\#\|u-\widetilde{u}\|_{L^2}\|z-\widetilde{z}\|_\mathcal{A}.
    \end{align*}
    So, 
    \[\|z-\widetilde{z}\|_\mathcal{A}\leq Hc_\#\|u-\widetilde{u}\|_{L^2}.\]
    
    \noindent
    Now, 
    \begin{align*}
        \|u-\widetilde{u}\|_{L^2}^2 = \mathcal{A}(z, u-\widetilde{u})
        = \mathcal{A}(z-\widetilde{z},u-\widetilde{u})
        \leq \overline{C}\|z-\widetilde{z}\|_{\mathcal{A}}\|u-\widetilde{u}\|_\mathcal{A}.
    \end{align*}
    Combining the results, the lemma is proved.
\end{proof}
\begin{remark}
    Following the similar lines of arguments, we can obtain
    \begin{equation}
        \|(u-\widetilde{u})_t\|_{L^2}\leq \overline{C} Hc_\#\Lambda^{-1/2}\|\widetilde{\kappa}^{-1/2}(f_t-u_{tt})\|_{L^2}.
        \label{eq: glob_t}
    \end{equation}
    Also note that $\widetilde{\kappa}^{-1/2}=O(H)$. So the $L^2$-error is also second order with respect to $H$.
\end{remark}

\subsubsection{Abstract Problem}
\label{section:abstract_t}

We now move on to the analysis of the corrector. The main idea is that the error propagation has an exponential decay with respect to the oversampling layers, similar to the time-independent case.
\begin{Problem}
    Let $K_i\in \mathcal{T}^H$ and $t_i(\cdot,t)$, $t_i^0\in V^\prime$ such that $\left<t_i,v\right>=0$ and $\left<t_i^0,v\right>=0$ for any $v\in V$ with $supp(v)\subset \Omega\backslash K_i$.
    Define $\mathcal{P}_i(\cdot,t):V^\prime\rightarrow V$ such that for all $v(\cdot,t)\in V$,
    \begin{equation}
    \begin{cases}
        (\Piti_t,v)+\mathcal{B}(\Piti,v)=\left<t_i,v\right>\\
        \mathcal{B}(\Piti(\cdot,0),v)=\left<t_i^0,v\right>,
        \label{eq:piti_t}
    \end{cases}
    \end{equation}
    and $\mathcal{P}^m_i:V^\prime\rightarrow V^m_i$ with
    \begin{equation}
        \begin{cases} 
            (\Pmiti_t,v)+\mathcal{B}(\Pmiti,v)=\left<t_i,v\right>\\
            \mathcal{B}(\Pmiti(\cdot,0),v)=\left<t_i^0,v\right>.
        \end{cases}
    \end{equation}
\end{Problem}

We aim to estimate 
    \[\left\|\sum^N_{i=1}\Piti-\Pmiti\right\|_\mathcal{E}^2 =\left\|\sum^N_{i=1}(\Piti-\Pmiti)(\cdot,T)\right\|_{L^2}^2+ \int^T_0\left\|\sum^N_{i=1}\Piti-\Pmiti\right\|_\mathcal{B}^2 \di t.\]

One should note that the initial condition here is exactly the abstract problem in the previous section. The results are carried over here.
In addition, we will further define two norms for our analysis,
\[
\|t_i\| = \max_{v\in V}\frac{\left<t_i,v\right>}{\|v\|_\mathcal{B}} \text{ and }
\|t^0_i\| = \max_{v\in V}\frac{\left<t^0_i,v\right>}{\|v\|_\mathcal{B}}.
\]
\begin{lemma}
\label{lemma:bound_operator}
Let $c_\#\coloneqq \beta_0^{-1}\kappa_1^{-1/2}\sqrt{1+\Lambda^{-1}}$. Then
    \begin{equation}
        \|\Piti(\cdot,0)\|_{L^2}^2\leq H^2 c_\#^2\|t^0_i\|^2,
    \end{equation}
    and
    \begin{equation}
        \|\Piti\|_\mathcal{E}^2\leq H^2 c_\#^2 \|t^0_i\|^2 + \int^T_0 \|t_i\|^2.
    \end{equation}
\end{lemma}
\begin{proof}
    
\begin{align*}
    \|\Piti(\cdot,0)\|_{L^2}^2&\leq H^2c_\#^2 \|\Piti(\cdot,0)\|_\mathcal{B}^2 \\
    &= H^2c_\#^2\left<t^0_i, \Piti(\cdot,0)\right>\\
    &\leq H^2c_\#^2 \|t^0_i\|\|\Piti(\cdot,0)\|_\mathcal{B}\\
    &\leq H^2 c_\#^2 \|t^0_i\|^2.
\end{align*}
Hence, 
\begin{equation}
    \|\Piti(\cdot,0)\|_{L^2}\leq Hc_\#\|t^0_i\|.
    \label{eq: piti(cdot,0)}
\end{equation}
Also, putting $v = \Piti$,
\begin{align*}
    (\Piti_t,\Piti)+\mathcal{B}(\Piti,\Piti)&=\left<t_i,\Piti\right>\\
    \int^T_0\frac{1}{2}\frac{\partial}{\partial t}\|\Piti_t\|_{L^2}^2
    +\|\Piti\|_\mathcal{B}^2 \di t &= \int^T_0 \left<t_i,\Piti\right>\\
    \frac{1}{2}\|\Piti(\cdot,T)\|_{L^2}^2 +\int^T_0 \|\Piti\|_\mathcal{B}^2&\leq \frac{1}{2}\|\Piti(\cdot,0)\|^2 + \frac{1}{2}\int^T_0 \|t_i\|^2 + \frac{1}{2}\int^T_0 \|\Piti\|_\mathcal{B}^2.\\
\end{align*}
Hence, by using equation \eqref{eq: piti(cdot,0)},
\begin{equation}
\begin{aligned}
    \|\Piti\|_\mathcal{E}^2\leq  H^2 c_\#^2 \|t^0_i\|^2 + \int^T_0 \|t_i\|^2.
\end{aligned}
    \label{eqn:E(Piti)}
\end{equation}
\end{proof}
\begin{lemma}
With the same notations in lemma \ref{lemma:part1},
\begin{equation}
        \|\Piti\|_{\mathcal{E}(\Omega \backslash K^m_i)}^2\leq\theta^m \left((m+1)H^2c_\#^2\|t^0_i\|^2+\int^T_0 \|t_i\|^2\right)
        \label{eq: Piti_E}.
    \end{equation}
    \label{lemma:part1_t}
\end{lemma}

\begin{proof}
Following the similar approach in the proof of lemma \ref{lemma:part1}, putting $v=(1-\chimi)\Piti$ into the equation \eqref{eq:piti_t} and then integrating over $t \in [0,T]$, we can obtain
\begin{align*}
    &\int^T_0(\Piti_t, \Piti)_{(\Omega\backslash K^m_i)} + \|\Piti\|_{\mathcal{B}{(\Omega\backslash K^m_i)}}^2\\
    &= \int^T_0(\Piti_t, (\chimi-1)\Piti)_{(K^m_i \backslash K^{m-1}_i)}
    +\int^T_0 \mathcal{B}_{(K^m_i \backslash K^{m-1}_i)}(\Piti, (\chimi-1)\Piti)\\
    &\leq \int_{K^m_i \backslash K^{m-1}_i} (\chimi-1) \int^T_0 \frac{1}{2}\frac{\partial}{\partial t}\|\Piti\|_{L^2(K^m_i \backslash K^{m-1}_i)}^2 + 
    \int^T_0 c_* \|\Piti\|_{\mathcal{B}(K^m_i \backslash K^{m-1}_i)}^2\\
    &\leq \frac{1}{2}\|\Piti(\cdot,0)\|_{L^2(K^m_i \backslash K^{m-1}_i)}^2+ 
    \int^T_0 c_* \|\Piti\|_{\mathcal{B}(K^m_i \backslash K^{m-1}_i)}^2.\\
\end{align*}
Then,
\begin{align*}
   &\|\Piti\|_{\mathcal{E}{(\Omega\backslash K^{m-1}_i)}}^2-\frac{1}{2}\|\Piti(\cdot,0)\|_{L^2(\Omega\backslash K^{m-1}_i)}^2\\
   &\geq \left(1 + \frac{1}{c_*}\right)\left(\|\Piti\|^2_{\mathcal{E}(\Omega\backslash K^m_i)}-\frac{1}{2}\|\Piti(\cdot,0)\|_{L^2(\Omega \backslash K^{m}_i)}^2\right)
   -\frac{1}{2c_*}\|\Piti(\cdot,0)\|_{L^2(K^m_i \backslash K^{m-1}_i)}^2\\
   &\geq \frac{1}{\theta}\left(\|\Piti\|^2_{\mathcal{E}(\Omega\backslash K^m_i)}-\frac{1}{2}\|\Piti(\cdot,0)\|_{L^2(\Omega \backslash K^{m}_i)}^2\right)-\frac{1}{2c_*}\|\Piti(\cdot,0)\|_{L^2(K^m_i \backslash K^{m-1}_i)}^2.\\
\end{align*}
Along with lemma \ref{lemma:part1}, this grants us 
    \begin{align*}
        \|\Piti\|_{\mathcal{E}(\Omega\backslash K^m_i)}^2
        &\leq \frac{1}{2}\|\Piti(\cdot,0)\|_{L^2(\Omega\backslash K^{m}_i)}^2 + \frac{\theta}{2c_*}\|\Piti(\cdot,0)\|_{L^2(K^m_i \backslash K^{m-1}_i)}^2
        +\theta\|\Piti\|_{\mathcal{E}(\Omega\backslash K^{m-1}_i)}^2\\
        &\leq \frac{1}{2}\theta^m\|\Piti(\cdot,0)\|_{L^2}^2 + \frac{\theta}{2c_*}\|\Piti(\cdot,0)\|_{L^2(\Omega\backslash K^{m-1}_i)}^2
        +\theta\|\Piti\|_{\mathcal{E}(\Omega\backslash K^{m-1}_i)}^2\\
        &\leq \frac{1}{2}\theta^m\left(1+\frac{1}{c_*}\right)\|\Piti(\cdot,0)\|_{L^2}^2 +\theta \|\Piti\|_{\mathcal{E}(\Omega\backslash K^{m-1}_i)}^2\\
        &\leq \frac{1}{2}\theta^{m-1}\|\Piti(\cdot,0)\|_{L^2}^2 +\theta \|\Piti\|_{\mathcal{E}(\Omega\backslash K^{m-1}_i)}^2\\
        &\leq \theta^m\left(m \|\Piti(\cdot,0)\|_{L^2}^2 + \|\Piti\|_\mathcal{E}^2\right)\\
        &\leq \theta^m\left(m H^2 c_\#^2\|\Piti(\cdot,0)\|_\mathcal{B}^2 + \|\Piti\|_\mathcal{E}^2\right)\\
        &\leq \theta^m\left(m H^2 c_\#^2\|t_i^0\|^2 + \|\Piti\|_\mathcal{E}^2\right).
    \end{align*}
Note that combined with equation \eqref{eqn:E(Piti)}, we obtain
    \[\|\Piti\|^2_{\mathcal{E}(\Omega\backslash K^m_i)}\leq  \theta^m \left((m+1)H^2c_\#^2\|t^0_i\|^2+\int^T_0 \|t_i\|^2\right).\]
\end{proof}

In parallel, we can have
\begin{equation}
    \|\Pmiti\|^2_{\mathcal{E}^2(\Omega\backslash K^m_i)}\leq  \theta^m \left((m+1)H^2c_\#^2\|t^0_i\|^2+\int^T_0 \|t_i\|^2\right).
    \label{eq: Pmiti_E}
\end{equation}

\begin{lemma}
    With the same notations in lemma \ref{lemma:part2}, we have
    \begin{equation}    
    \begin{aligned}
    \|\Piti-\Pmiti\|_\mathcal{E}^2\leq 2\theta^{m-1}(1+2\overline{C}^2 c_\star) \left[H^2 c_\#^2 (m+2)\|t_i^0\|^2+\int^T_0\|t_i\|^2\right].
    \end{aligned}
    \end{equation}
    \label{lemma:part2_t}
\end{lemma}
\begin{proof}
    Recall that $z_i = \Piti-\Pmiti$ and
    $$z_i = \left[(1-\chimi)\Piti\right]+\left[(\chimi-1)\Pmiti+\chimi z_i\right] =: z_i^\prime + z_i^{\prime\prime}.$$
    With equations \eqref{eq: Pmiti_E} and \eqref{eq: Piti_E},
    \[
    \begin{aligned}
    \|z_i(\cdot,T)\|_{\mathcal{E}(\Omega\backslash K^{m-1}_i)}^2\leq  \theta^{m-1}\left(mH^2c_\#^2\|t^0_i\|^2+\int^T_0 \|t_i\|^2\right).
    \end{aligned}
    \]    
    and by lemmas \ref{lemma:bound_operator} and \ref{lemma:part2},
    \begin{align*}
    \|z_i(\cdot,0)\|^2_{L^2(\Omega\backslash K^{m-1}_i)}
    &\leq \overline{C}^2 c_\star H^2c_\#^2\theta^{m-1} \|\Piti(\cdot,0)\|_\mathcal{B}^2\\
    &\leq \overline{C}^2 c_\star H^2c_\#^2\theta^{m-1} \|t_i^0\|^2,
    \end{align*}
    and by lemma \ref{lemma:part1},
    \begin{align*}
        \|\Piti(\cdot,0)\|_{L^2(\Omega\backslash K^{m-1})}^2
        \leq\overline{C}^2 c_\star H^2c_\#^2\theta^{m-1}\|t_i^0\|^2.
    \end{align*}
    Then,
    \begin{align*}
        &\qquad \|z_i\|^2_\mathcal{E}-\frac{1}{2}\|z_i(\cdot,0)\|^2\\
        &\leq \int^T_0((z_i)_t,z_i^\prime)+\int^T_0\mathcal{B}(z_i,z_i^\prime)\\
        &\leq \left.(z_i,z_i^\prime)\right|^T_0 - \int^T_0 ((z^\prime_i)_t,z_i) + \overline{C}\int^T_0 \|z_i\|_{\mathcal{B}{(\Omega\backslash K^{m-1}_i)}}
        \|z_i^\prime\|_{\mathcal{B}{(\Omega\backslash K^{m-1}_i)}}\\
        &\leq \|z_i(\cdot,T)\|_{L^2(\Omega\backslash K^{m-1}_i)}\|z^\prime_i(\cdot,T)\|_{L^2(\Omega\backslash K^{m-1}_i)}
        + \|z_i(\cdot,0)\|_{L^2(\Omega\backslash K^{m-1}_i)}\|z^\prime_i(\cdot,0)\|_{L^2(\Omega\backslash K^{m-1}_i)}\\
        &\qquad +\int^T_0 \mathcal{B}_{(\Omega\backslash K^{m-1}_i)}(\Piti,(1-\chimi)z_i) + \overline{C}\int^T_0\|z_i\|_{\mathcal{B}{(\Omega\backslash K^{m-1}_i)}}
        \|z_i^\prime\|_{\mathcal{B}{(\Omega\backslash K^{m-1}_i)}}\\
        &\leq \theta^{m-1}\left(mH^2c_\#^2\|t^0_i\|^2+\int^T_0\|t_i\|^2\right)+\left(\theta^{m-1}H^2c_\#^2\|t^0_i\|^2\right)\\
        &\qquad + \overline{C}\|\Piti\|_{\mathcal{E}(\Omega\backslash K^{m-1}_i)}\left(\int^T_0\left\|(1-\chimi)z_i\right\|^2_{\mathcal{B}(\Omega \backslash K^{m-1}_i)}\right)^{1/2}\\
        &\qquad +\overline{C}\sqrt{c_\star}\|z_i\|_{\mathcal{E}(\Omega\backslash K^{m-1}_i)}\left(\int^T_0\|\Piti\|_{\mathcal{B}(\Omega\backslash K^{m-1}_i)}^2\right)^{1/2}\\
        &\leq \theta^{m-1}\left((m+1)H^2 c_\#^2 \|t^0_i\|^2+\int^T_0 \|t_i\|^2\right)
         + 2\overline{C}\sqrt{c_\star}\|\Piti\|_{\mathcal{E}(\Omega\backslash K^{m-1}_i)}\|z_i\|_{\mathcal{E}(\Omega\backslash K^{m-1}_i)}\\
        &\leq \theta^{m-1}\left((m+1)H^2c_\#^2\|t^0_i\|^2+\int^T_0\|t_i\|^2\right)+2\overline{C}^2c_\star \|\Piti\|_{\mathcal{E}(\Omega \backslash K^{m-1}_i)}^2 + \frac{1}{2}\|z_i\|_{\mathcal{E}}^2.
    \end{align*}
    Hence,
    \begin{align*}
        \|z_i\|_\mathcal{E}^2\leq 2\theta^{m-1}(1+2\overline{C}^2 c_\star) \left[H^2 c_\#^2 (m+2)\|t_i^0\|^2+\int^T_0\|t_i\|^2\right].
    \end{align*}
\end{proof}

\begin{lemma}
    With the same notations in lemma \ref{lemma:part3}, we have
    \begin{equation}
    \begin{aligned}
        \left\|\sum^N_{i=1}\Piti-\Pmiti\right\|_\mathcal{E}^2
        \leq \theta^{m-1}\overline{C}(2+\overline{C}^2 c_\star) \left[H^2 c_\#^2 (m+2)\sum^N_{i=1}\|t_i^0\|^2+\int^T_0\sum^N_{i=1}\|t_i\|^2\right].
    \end{aligned}
    \end{equation}
    \label{lemma:part3_t}
\end{lemma}
\begin{proof}
    With the same notation, $z = \sum^N_i z_i$.
    First, with lemma \ref{lemma:part2_t},
    \begin{align*}
        &\quad \left\|\sum^N_{i=1}z_i(\cdot,T)\right\|_{L^2}^2\leq \left(\sum^N_{i=1}\|z_i(\cdot,T)\|_{L^2}\right)^2\\
        &\leq \left(\sum^N_{i=1}\theta^{\frac{m-1}{2}}(2+\overline{C}^2 c_\star)^{1/2} \left[H^2 c_\#^2 (m+2)\|t_i^0\|^2+\int^T_0\|t_i\|^2\right]^{1/2}\right)^2\\
        &\leq \theta^{m-1}(2+\overline{C}^2 c_\star)\left(H^2 c_\#^2 (m+2)\sum^N_{i=1}\|t_i^0\|^2+\int^T_0\sum^N_{i=1}\|t_i\|^2\right).
    \end{align*}
    Next, by the Cauchy Schwartz inequality, 
    \begin{align*}
        &\quad \int^T_0 \|z\|_\mathcal{B}^2 =\int^T_0 \sum_{i,j}\mathcal{B}(z_i,z_j)\leq \overline{C}\int^T_0\sum_{i,j}\|z_i\|_\mathcal{B}\|z_j\|_\mathcal{B} \leq \overline{C}{\sum^N_{i=1}\int^T_0\|z_i\|_\mathcal{B}^2}\\
        &\leq \theta^{m-1}\overline{C}\sum^N_{i=1}(2+\overline{C}^2 c_\star) \left[H^2 c_\#^2 (m+2)\|t_i^0\|^2+\int^T_0\|t_i\|^2\right]\\
        &\leq \theta^{m-1}\overline{C}(2+\overline{C}^2 c_\star) \left[H^2 c_\#^2 (m+2)\sum^N_{i=1}\|t_i^0\|^2+\int^T_0\sum^N_{i=1}\|t_i\|^2\right].
    \end{align*}
    Now, together with all the terms, the proof is complete.
\end{proof}

\begin{theorem}
Suppose $\mathcal{D}^m \widetilde{g}$, $\mathcal{N}^m q$, $w^m$, $\mathcal{N}^\mathup{glo} q$, $\mathcal{D}^\mathup{glo} \widetilde{g}$ and $w^\mathup{glo}$ are constructed by equations \eqref{eq:dg_time_variant_ms}, \eqref{eq:nq_time_variant_ms}, \eqref{eqn:ms_t_weakBVP}, \eqref{eq:dg_time_variant}, \eqref{eq:nq_time_variant}, and \eqref{eq:elliptic_projection_t} respectively. Suppose $\overline{C}$, $\theta$, $c_\#$, $c_\star$ and $C_\mathup{inv}$ are defined as in lemmas \ref{lemma:overline_C}, \ref{lemma:part1}, \ref{lemma:part2} and \ref{lemma:Cinv} respectively. Then, 
    \begin{equation}
        \begin{aligned}   
        \|\mathcal{D}^\mathup{glo}\widetilde{g}-\mathcal{D}^m \widetilde{g}\|^2_\mathcal{E}
        \leq \theta^{m-1}\overline{C}^2(2+\overline{C}^2 c_\star) \left[H^2 c_\#^2 (m+2)\|\widetilde{g}(\cdot,0)\|_{L^2}^2+\int^T_0(\|\widetilde{g}\|_\mathcal{A}^2+\|\widetilde{g}_t\|_{L^2}^2)\right];
        \end{aligned}
    \end{equation}
    \begin{equation}
        \begin{aligned}   
        \|\mathcal{N}^\mathup{glo}q-\mathcal{N}^m q\|^2_\mathcal{E}
        \leq \theta^{m-1}C_\mathup{tr}^2(2+\overline{C}^2 c_\star) \left[H^2 c_\#^2 (m+2)\|q(\cdot,0)\|_{L^2(\Gamma_N)}^2+\int^T_0\|q\|_{L^2(\Gamma_N)}^2\right];
        \end{aligned}
    \end{equation}
    and
    \begin{equation}
    \begin{aligned}
        &\|{w}^\mathup{glo}-w^m_*\|^2_\mathcal{E}\\
        &\leq
        \theta^{m-1}\overline{C}(2+\overline{C}^2c_\star)\max\left(C_\mathup{inv}^2,1\right)\left\{H^2c_\#^2\left[(m+2)\|{w}^\mathup{glo}(\cdot,0)\|_{\mathcal{B}}^2+\int^T_0\|{w}^\mathup{glo}_t\|_{L^2}^2\right]+\overline{C}^2\int^T_0\|{w}^\mathup{glo}\|_\mathcal{B}^2\right\}.
    \end{aligned}
    \end{equation}
\end{theorem}
\begin{proof}
    For the Dirichlet condition,
    \begin{align*}
    \left<t_i,\mathcal{D}^\mathup{glo}\widetilde{g}\right>&=\mathcal{A}_{(K_i)}(\widetilde{g},\mathcal{D}^\mathup{glo}\widetilde{g})+ (\widetilde{g}_t,\mathcal{D}^\mathup{glo}\widetilde{g})_{(K_i)}\\
        &\leq \left(\overline{C}\|\widetilde{g}\|_{\mathcal{A}(K_i)}+Hc_\#\|\widetilde{g}_t\|_{L^2(K_i)}\right)\|\mathcal{D}^\mathup{glo}\widetilde{g}\|_\mathcal{B}.
    \end{align*}
    Hence, \begin{align*}    
    \sum^N_{i=1}\|t_i\|^2&\leq \sum^N_{i=1}\left(\overline{C}\|\widetilde{g}\|_{\mathcal{A}(K_i)}+Hc_\#\|\widetilde{g}_t\|_{L^2(K_i)}\right)^2\\
    &=\left(\overline{C}\sqrt{\sum^N_{i=1}\|\widetilde{g}\|^2_{\mathcal{A}(K_i)}}+Hc_\#\sqrt{\sum^N_{i=1}\|\widetilde{g}_t\|_{L^2(K_i)}^2}\right)^2\\
    &=(\overline{C}\|\widetilde{g}\|_\mathcal{A}+Hc_\#\|\widetilde{g}_t\|_{L^2})^2.
    \end{align*}
    Similarly, 
    \[\sum^N_{i=1}\|t^0_i\|^2\leq \sum^N_{i=1}\overline{C}^2\|\widetilde{g}(\cdot,0)\|_{L^2}^2.\]
    Now for the Neumman corrector,
    \begin{align*}
        \sum^N_{i=1}\int_{\Gamma_N\cap\partial K_i}q\mathcal{N}^\mathup{glo}_i q \di \sigma 
        &\leq \sum^N_{i=1}\|q\|_{L^2(\Gamma_N\cap\partial K_i)}\|\mathcal{N}^\mathup{glo}_i q\|_{L^2(\Gamma_N)}\\
        &\leq \sum^N_{i=1} C_\mathup{tr}\|q\|_{L^2(\Gamma_N\cap\partial K_i)}\|\mathcal{N}^\mathup{glo}_i q\|_{\mathcal{A}}\\
        &\leq C_\mathup{tr}^2 \|q\|_{L^2(\Gamma_N)}^2.
    \end{align*}
    Finally, we deal with ${w}^\mathup{glo}-w^m_*$. Note that $s(\pi_i\varphi_*,\pi v)\leq \|\pi_i\varphi_*\|_s\|v\|_\mathcal{B}$ at any $t>0$ for $v\in V$. To bound $\|t_i\|$, it suffices to bound $\|\pi_i\varphi_*(\cdot,t)\|_s$.
    Akin to the time-independent case, for each $t>0$, there exists $\widehat{\varphi}_*(\cdot,t)\in V$ such that $\pi\widehat{\varphi}_*=\pi\varphi_*$ and $\|\widehat{\varphi}_*\|_\mathcal{A}\leq C_\mathup{inv}\|\pi \varphi_*\|_s$.
    Then,
    \begin{align*}
        \|\pi\varphi\|_s^2 &= (w^\mathup{glo}_t,\widehat{\varphi}_*)+\mathcal{B}(w^\mathup{glo},\widehat{\varphi}_*)\\
        &\leq \|w^\mathup{glo}_t\|_{L^2}\|\widehat{\varphi}_*\|_{L^2}+\overline{C}\|w^\mathup{glo}\|_\mathcal{B}\|\widehat{\varphi}_*\|_\mathcal{B}\\
        &\leq \left(Hc_\#\|w^\mathup{glo}_t\|_{L^2}+\overline{C}\|w^\mathup{glo}\|_\mathcal{B}\right)\|\widehat{\varphi}_*\|_\mathcal{B}\\
        &\leq \max\left(C_\mathup{inv},1\right)\left(Hc_\#\|w^\mathup{glo}_t\|_{L^2}+\overline{C}\|w^\mathup{glo}\|_\mathcal{B}\right)\|\pi{\varphi_*}\|_s.
    \end{align*}
    Hence, by assembling all the terms, one can obtain the desired result. 
\end{proof}

\begin{Corollary}
If $\widetilde{g}_{tt}$, $q_t$ and $u_{tt}$ exist in $L^2(\Omega, (0,T))$, then
\begin{align*}
        \|(\mathcal{D}^\mathup{glo}\widetilde{g}-\mathcal{D}^m \widetilde{g})_t\|_\mathcal{E}^2
        \leq \theta^{m-1}\overline{C}^2(2+\overline{C}^2 c_\star) \left[H^2 c_\#^2 (m+2)\|\widetilde{g}_t(\cdot,0)\|_{L^2}^2+\int^T_0(\|\widetilde{g}_t\|_\mathcal{A}^2+\|\widetilde{g}_{tt}\|_{L^2}^2)\right];
\end{align*} 
\begin{align*}
        \|(\mathcal{N}^\mathup{glo} q-\mathcal{N}^m q)_t\|_\mathcal{E}^2 
        \leq \theta^{m-1}C_\mathup{tr}^2(2+\overline{C}^2 c_\star) \left[H^2 c_\#^2 (m+2)\|q_t(\cdot,0)\|_{L^2(\Gamma_N)}^2+\int^T_0\|q_t\|_{L^2(\Gamma_N)}^2\right];
\end{align*}
and
\begin{align*}
        &\|({w}^\mathup{glo}-w^m_*)_t\|^2_\mathcal{E}\\
        &\leq
        \theta^{m-1}\overline{C}(2+\overline{C}^2c_\star)\max\left(C_\mathup{inv}^2,1\right)\left\{H^2c_\#^2\left[(m+2)\|{w}^\mathup{glo}_t(\cdot,0)\|_{\mathcal{B}}^2+\int^T_0\|{w}^\mathup{glo}_{tt}\|_{L^2}^2\right]+\overline{C}^2\int^T_0\|{w}^\mathup{glo}_t\|_\mathcal{B}^2\right\}.
\end{align*}
\end{Corollary}
\begin{proof}
    Now, by taking derivatives with respect to time, we can have another $t_i^\prime(\cdot,t), (t^0_i)^\prime \in V^\prime$ such that 
    \begin{equation*}
        \begin{cases}
            (\Piti_{tt},v)+\mathcal{B}(\Piti_t,v)=\left<t_i^\prime,v\right>\\
            \mathcal{B}(\Piti_t(\cdot,0),v)=\left<(t^0_i)^\prime,v\right>.
        \end{cases}
    \end{equation*}
    Then, following the same lines of arguments of the abstract problems obtains the desired results. The existence of the time derivative of the corrector follows from the regularity of the boundary value functions.
\end{proof}
\begin{Corollary}
    If furthermore $C_\mathup{inv}\theta^{(m-1)/2}(m+2)^{d/2}=O(H^2)$, then
    \begin{align*}
        \|\mathcal{D}^\mathup{glo}\widetilde{g}-\mathcal{D}^m \widetilde{g}\|_\mathcal{E}&\leq O(H^2+\sqrt{T}H),\\
        \|\mathcal{N}^\mathup{glo} q-\mathcal{N}^m q\|_\mathcal{E}&\leq O(H^2+\sqrt{T}H),\\
        \|w^\mathup{glo}-w^m_*\|_\mathcal{E}&\leq O(H^2+\sqrt{T}H).
    \end{align*}
\end{Corollary}
\begin{proof}The cases for $\mathcal{D}^m\widetilde{g}$ and $\mathcal{N}^{m} q $ are clear. For $w^\mathup{glo}$,
    by utilizing equations \eqref{eq:dg_time_variant} and \eqref{eq:nq_time_variant}, with the analysis in the time-independent case, it is easy to see $\int^T_0\|w^\mathup{glo}\|_\mathcal{B}^2 \lesssim \int^T_0 \|u_0\|_\mathcal{B}^2+\int^T_0 \|\mathcal{D}^\mathup{glo}\widetilde{g}\|_\mathcal{B}^2+\int^T_0 \|\mathcal{N}^\mathup{glo}q\|_\mathcal{B}^2$ $=O(T(H^{-2}+1))$. 
    On the other hand, 
    \begin{align*}
        \int^T_0 \|\mathcal{D}^\mathup{glo}\widetilde{g}_t\|_{L^2}^2
        &\lesssim H^2 \int^T_0 \|\mathcal{D}^\mathup{glo}\widetilde{g}_t\|_{\mathcal{B}}^2\\
        &\lesssim H^2\left\{\|\mathcal{D}^\mathup{glo}\widetilde{g}_t(\cdot,0)\|_{L^2}^2 + \int^T_0\|\widetilde{g}_t\|_\mathcal{A}^2+H^2\int^T_0\|\widetilde{g}_{tt}\|_{L^2}^2\right\}\\
        &=O(H^2+TH^2 + TH^4).
    \end{align*}
    Similarly, $\int^T_0 \|\mathcal{N}^\mathup{glo}q_t\|_{L^2}^2=O(H^2+TH^2)$. Then, 
    \[
    \int^T_0 \|w^\mathup{glo}_t\|_{L^2}^2\lesssim \int^T_0 \|\partial_t u_0\|_{L^2}^2+\int^T_0 \|\mathcal{D}^\mathup{glo}\widetilde{g}_t\|_{L^2}^2+\int^T_0 \|\mathcal{N}^\mathup{glo}q_t\|_{L^2}^2=O(T+H^2+TH^2).
    \] 
    Also, by the arguments in the time-independent case, $\|w^\mathup{glo}(\cdot,0)\|_\mathcal{B}=O(H^{-1})$. Combining all the terms, we have $\|w^\mathup{glo}-w^m_*\|_\mathcal{E}=O(H^2+\sqrt{T}H)$.
\end{proof}
\subsubsection{Multiscale Solutions}

We will first bound $\int^T_0\|u-u^\mathup{ms}\|_{L^2}^2$.
\begin{lemma} Suppose $u^\mathup{ms}$ is constructed by equation \eqref{eq:time_dpdt_solution} and $u$ is the actual solution of equation \eqref{eqn:unsteadyweakBVP}. Then
\begin{align*}
    \int^T_0 \|u-u^\mathup{ms}\|_{L^2}^2
    &\leq4T\|(u-u^\mathup{ms})(\cdot,0)\|_{L^2}^2\\
    &+2H^2c_\#^2(1+T^2)
    \int^T_0\left\{\|\widetilde{\kappa}^{-1/2}(f-u_t)\|^2_{L^2}
    +\|\widetilde{\kappa}^{-1/2}(f_t-u_{tt})\|^2_{L^2}\right\}.
\end{align*}
\end{lemma}
\begin{proof}
    Let $u-u^\mathup{ms} = (u-\widetilde{u})+(\widetilde{u}-u^\mathup{ms})=: \vartheta+\rho$.
    Note that \[\|\vartheta\|_{L^2}\leq H\Lambda^{-1/2} \overline{C}c_\#\|\widetilde{\kappa}^{-1/2}(f-u_t)\|_{L^2}.\]
    For $\rho$,
    \begin{align*}
        ((\widetilde{u}-u^\mathup{ms})_t,v) + \mathcal{A}(\widetilde{u}-u^\mathup{ms},v)
        = (\widetilde{u}_t,v)+\mathcal{A}(u,v)-(f,v)=((\widetilde{u}-u)_t,v).
    \end{align*}
    Putting $v=\rho$,
    \begin{align*}
        (\rho_t,\rho)+\|\rho\|_\mathcal{A}^2 &= ((u-\widetilde{u})_t,\rho)\\
        \|\rho\|_{L^2}\frac{\partial}{\partial t}\|\rho\|_{L^2}=\frac{1}{2}\frac{\partial}{\partial t}\|\rho\|^2_{L^2}&\leq \|(u-\widetilde{u})_t\|_{L^2}\|\rho\|_{L^2}.
    \end{align*}
    Hence, for any $t>0$,
    \begin{align*}
        \|\rho(\cdot,t)\|_{L^2}\leq \|\rho(\cdot,0)\|_{L^2}+\int^t_0\|(u-\widetilde{u})_t(\cdot,s)\|_{L^2}\di s.
    \end{align*}
    All in all,
    \begin{align*}
        &\quad \int^T_0\|u-u^\mathup{ms}\|_{L^2}^2
        =2 \int^T_0 \|\vartheta(\cdot,t)\|_{L^2}^2+\|\rho(\cdot,t)\|_{L^2}^2
        \\
        &\leq 2H^2\Lambda^{-1}\overline{C}^2c_\#^2
        \int^T_0 \|\widetilde{\kappa}^{-1/2}(f-u_t)\|_{L^2}^2
        +2\int^T_0 \left\{\|\rho(\cdot,0)\|_{L^2}+\int^t_0\|(u-\widetilde{u})_t(\cdot,s)\|\di s\right\}^2 \di t\\
        &\leq2H^2\Lambda^{-1}\overline{C}^2c_\#^2
        \int^T_0 \|\widetilde{\kappa}^{-1/2}(f-u_t)\|_{L^2}^2
        +4\int^T_0\left[\|\rho(\cdot,0)\|_{L^2}^2
        +\int^t_0 t\|(u-\widetilde{u})_t(\cdot,s)\|^2 \di s \right]\di t\\
        &\leq 2H^2\Lambda^{-1}\overline{C}^2c_\#^2
        \int^T_0 \|\widetilde{\kappa}^{-1/2}(f-u_t)\|_{L^2}^2
        +4\int^T_0\|\rho(\cdot,0)\|_{L^2}^2
        +4\int^T_0 t(T-t)\|(u-\widetilde{u})_t\|^2 \di t\\
        &\leq 2H^2\Lambda^{-1}\overline{C}^2c_\#^2
        \int^T_0 \|\widetilde{\kappa}^{-1/2}(f-u_t)\|_{L^2}^2
        +4\int^T_0\|\rho(\cdot,0)\|_{L^2}^2
        +T^2\int^T_0\|(u-\widetilde{u})_t\|^2 \di t.
    \end{align*}
    The desired results are then obtained by using equation \eqref{eq: glob_t}.
\end{proof}
\begin{Corollary}
    If furthermore $C_\mathup{inv}\theta^{(m-1)/2}(m+2)^{d/2}=O(H^2)$, then
    \[
    \int^T_0\|u-u^\mathup{ms}\|^2_{L^2}\leq O((T+T^3)H^4)\text{ and }\|(u-u^\mathup{ms})(\cdot,T)\|_{L^2}=O(H^2+\sqrt{T}H^2).
    \]
\end{Corollary}

We can now combine all the results to bound the terms in equation \eqref{eq:main_t}.

\begin{theorem}
    If $\widetilde{g}_{tt}$, $q_t$ and $u_{tt}\in L^2(\Omega, (0,T))$ exist on $(0,T)$ and $C_\mathup{inv}\theta^{(m-1)/2}(m+2)^{d/2}=O(H^2)$, then
    \begin{equation*}
        \int^T_0\|u-u^\mathup{ms}\|_\mathcal{A}^2\leq O(H^4+(T+T^2) H^2)
    \end{equation*}
\end{theorem}
\begin{proof}
    Recall from equation \eqref{eq:main_t}, using the results from the time-independent problem, we take $v=w^m_*-\mathcal{D}^m\widetilde{g}+\mathcal{N}^m q+\widetilde{g}$. Then,
    \begin{align*}
        \frac{1}{4}(i) = \|(u-u^\mathup{ms})(\cdot,0)\|_{L^2}^2
        \leq O(H^4).
    \end{align*}
    \begin{align*}
        \frac{1}{2}(ii) &= \|(u-v)(\cdot,0)\|_{L^2}^2
    \leq H^2c_\#^2 \|(u-v)(\cdot,0)\|_\mathcal{B}^2\\
        &\leq H^2c_\#^2 \left\{\|(u-\widetilde{u})(\cdot,0)\|_\mathcal{A}+\|(w^\mathup{glo}-w^m_{*})(\cdot,0)\|_\mathcal{B}\right.\\
        &+\left.\|(\mathcal{D}^\mathup{glo}\widetilde{g}-\mathcal{D}^m \widetilde{g})(\cdot,0)\|_\mathcal{B}
        +\|(\mathcal{N}^\mathup{glo} q-\mathcal{N}^m q)(\cdot,0)\|_\mathcal{B}\right\}^2\\
        &\leq O(H^4).
    \end{align*}
    \begin{align*}
        &\quad \frac{1}{4\overline{C}^2}(iii)
        =\|u-v\|_\mathcal{E}^2\\
        &\leq \left\{\|u-\widetilde{u}\|_\mathcal{E}
        +\|w^\mathup{glo}-w^m_*\|_\mathcal{E}
        +\|\mathcal{D}^\mathup{glo}\widetilde{g}-\mathcal{D}^m \widetilde{g}\|_\mathcal{E}
        +\|\mathcal{N}^\mathup{glo} q-\mathcal{N}^m q\|_\mathcal{E}\right\}^2\\
        &\leq \left\{
        O(H^2+\sqrt{T}H^2)+ O(H^2+\sqrt{T}H)
        +O(H^2+\sqrt{T}H)+O(H^2+\sqrt{T}H)
        \right\}^2\\
        &\leq O(H^4+TH^2).
    \end{align*}
    \begin{align*}
        \frac{1}{4}(iv)
        &=\sqrt{\int^T_0\|u-u^\mathup{ms}\|_{L^2}^2}\sqrt{\int^T_0\|(u-v)_t\|_{L^2}^2}\\
        &\leq O(\sqrt{T+T^3}H^2)\left\{\int^T_0\left[
        \|(u-\widetilde{u})_t\|_{L^2}
        +\|(w^\mathup{glo}-w^m_*)_t\|_{L^2}
        +\|(\mathcal{D}^\mathup{glo}\widetilde{g}-\mathcal{D}^m \widetilde{g})_t\|_{L^2}\right.\right.\\
        &\left.\left.+\|(\mathcal{N}^\mathup{glo} q-\mathcal{N}^m q)_t\|_{L^2}
        \right]^2\right\}^{1/2}\\
        &\leq O(\sqrt{T+T^3}H^2)O(H^3+\sqrt{T}H^2)\\
        &= O((T+T^2)H^4).
    \end{align*}
    Altogether,
    \begin{align*}
       \|(u-u^\mathup{ms})(\cdot,T)\|_{L^2}^2+\int^T_0\|u-u^\mathup{ms}\|_\mathcal{A}^2\leq O(H^4+(T+T^2)H^2).
    \end{align*}
\end{proof}

\subsection{Temporal Discretization}

We apply the Backward Euler method to the scheme. Explicitly, we let $\tau$ be the time step and $U^n_\mathup{ms}\coloneqq u^\mathup{ms}(t_n)$.
Then, the variational formulation \eqref{eqn:ms_t_weakBVP} becomes
\begin{equation}
    \begin{aligned}
    \left(\frac{U^n_\mathup{ms}-U^{n-1}_\mathup{ms}}{\tau},v\right) + \mathcal{A}(U^n_\mathup{ms},v) = (f(t_n),v) &\text{ for } v\in V^m_\mathup{ms},\\
    U^0_\mathup{ms} = u^0_\mathup{ms}. &
    \end{aligned}
    \label{eqn:beuler_weakBVP}
\end{equation}
\subsubsection{Backward Euler Schemes}

We compare two versions of the Backward Euler Scheme for convection diffusion, using the Dirichlet boundary condition as an illustration. The Neumann and Robin conditions follow similar treatments. Let $\tau$ be the timestep. 
\begin{enumerate}
    \item \begin{equation}
        \frac{u^{n+1}-u^n}{\tau}+ (\bbeta\cdot\nabla u^{n+1})=\nabla \cdot (A \nabla u^{n+1})
    \end{equation}
    \item \begin{equation}
        \frac{u^{n+1}-u^n}{\tau}+ (\bbeta\cdot\nabla u^{n})=\nabla \cdot (A \nabla u^{n+1}).
    \end{equation}
\end{enumerate}

Respectively, the application of the method becomes: given the multiscale solution $u^{n}$ at step $n$,
\begin{enumerate}
    \item Convection-Diffusion approach (CD-approach)
    \begin{enumerate}
        \item  Find $\mathcal{D}^m_i \widetilde{g}^{n+1}\in V^{m}_\mathup{ms}$ such that for $v\in V^{m}_\mathup{ms}$ such that
        \begin{equation*}
            \begin{aligned}
                \left(\frac{\mathcal{D}^m_i \widetilde{g}^{n+1}-\mathcal{D}^m_i \widetilde{g}^n}{\tau},v\right) + \mathcal{B}(\mathcal{D}^m_i \widetilde{g}^{n+1},v)=\mathcal{A}_{(K_i)}(\widetilde{g}^{n+1},v)+(\widetilde{g}_t^{n+1},v)_{(K_i)}.
            \end{aligned}
        \end{equation*}
        Then, set $\mathcal{D}^m \widetilde{g}^{n+1} = \sum^N_{i=1}\mathcal{D}^m_i \widetilde{g}^{n+1}$.
        \item Find $w^{n+1}\in V^{m}_\mathup{ms}$ such that for $v\in V^m_\mathup{ms}$, 
        \begin{equation*}
            \begin{aligned}
                \left(\frac{w^{n+1}-w^{n}}{\tau},v\right)+\mathcal{A}(w^{n+1},v)
                &= (f^{n+1},v)-\mathcal{A}(\widetilde{g}^{n+1},v)-(\widetilde{g}^{n+1}_t,v)\\
                &+\mathcal{A}(\mathcal{D}^m \widetilde{g}^{n+1},v)+(\mathcal{D}^m \widetilde{g}^{n+1}_t,v).
            \end{aligned}
        \end{equation*}
        \item Set $u_\mathup{ms}^{n+1}=w^{n+1}-\mathcal{D}^m \widetilde{g}^{n+1}+\widetilde{g}^{n+1}$.
    \end{enumerate}
    \item Diffusion approach (D-approach)
    \begin{enumerate}
        \item Find $\mathcal{D}^m_i \widetilde{g}^{n+1}\in V^{m}_\mathup{ms}$ such that for $v\in V^{m}_\mathup{ms}$ such that
        \begin{equation*}
            \begin{aligned}
                \left(\frac{\mathcal{D}^m_i \widetilde{g}^{n+1}-\mathcal{D}^m_i \widetilde{g}^n}{\tau},v\right) + a(\mathcal{D}^m_i \widetilde{g}^{n+1},v)+s(\pi \mathcal{D}^m_i \widetilde{g}^{n+1},\pi v)=a_{(K_i)}(\widetilde{g}^{n+1},v)+(\widetilde{g}_t^{n+1},v)_{(K_i)}.
            \end{aligned}
        \end{equation*}
        Then, set $\mathcal{D}^m \widetilde{g}^{n+1} = \sum^N_{i=1}\mathcal{D}^m_i \widetilde{g}^{n+1}$.
        \item Find $w^{n+1}\in V^{m}_\mathup{ms}$ such that for $v\in V^m_\mathup{ms}$,
        \begin{equation*}
\begin{aligned}
                \left(\frac{w^{n+1}-w^{n}}{\tau},v\right)+a(w^{n+1},v)
                &= (f^{n+1},v)-a(\widetilde{g}^{n+1},v)-(\widetilde{g}^{n+1}_t,v)\\
                &-(\bbeta\cdot\nabla u^n,v)+a(\mathcal{D}^m \widetilde{g}^{n+1},v)+(\mathcal{D}^m \widetilde{g}^{n+1}_t,v).
\end{aligned}
        \end{equation*}
        \item Set $u_\mathup{ms}^{n+1}=w^{n+1}-\mathcal{D}^m \widetilde{g}^{n+1}+\widetilde{g}^{n+1}$.
    \end{enumerate}
\end{enumerate}

\subsection{Nonlinear Convection Diffusion IBVPs}

In this subsection, we are interested in the convection diffusion with a nonlinearity term $f(u)$:
\begin{equation}
    \begin{cases}
        \partial_t u + \bbeta\cdot\nabla u = f(u) + \nabla\cdot(\boldsymbol{A}u), &\text{ in } \Omega\times[0,T],\\
        u = g, &\text{ on }\Gamma_D\times[0,T],\\
        b u + \bnu\cdot(\boldsymbol{A}\nabla u -\bbeta u) = q, &\text{ on }\Gamma_N\times[0,T],\\
        u(\cdot,0) = u_\mathup{init}, &\text{ on }\Omega.
    \end{cases}
\end{equation}

The traditional approach is to perform operator splitting. It is to decompose the convection-diffusion operator into two sub-problems, each targeting one operator \cite{CONNORS2014181}. 
However, with inhomogeneous boundary conditions, some operators are not left-invariant \cite{doi:10.1137/140994204}. Therefore, a correction term $\eta_n$ needs to be introduced.
On the other hand, with the current scheme on convection diffusion equations, the convection and diffusion operators can be considered at once. In other words, our ultimate goal is to construct the solution at the next step via:
\[
u^{n+1} = (\mathcal{S}^{\tau/2}_{f-\eta_n}\circ \mathcal{S}^\tau_{CD+\eta_n}\circ \mathcal{S}^{\tau/2}_{f-\eta_n})(u^n)
\]
where $\mathcal{S}^{\tau/2}_{f-\eta_n}$ maps to the solution considering the adjusted nonlinearity term $f-\eta_n$ for a stepsize $\tau/2$ and $\mathcal{S}^\tau_{CD+\eta_n}$ maps to the solution considering the convection diffusion equation with the source term $\eta_n$. 
Also, this corrector term $\eta_n$ is to be carefully selected and get cancelled out in the process; otherwise can be accumulated over time and become a stiff term. The problem becomes even more apparent with time-variant boundary conditions. In light of this, we will decompose the nonlinearity into two parts, one dependent on the boundary and one independent. We will illustrate this idea using time-invariant Dirichlet boundary conditions. The time-variant version extends naturally by following a similar fashion.
The proof of this well-known method can be found in \cite{doi:10.1137/140994204,doi:10.1137/16M1056250,doi:10.1137/19M1257081}. To summarise: at each step $n$,
\begin{enumerate}
    \item Define $z\in L^2(\Omega)$ such that $\left.z\right|_{\partial \Omega}=g(\cdot,t_n)$.
    \item Find $w(\frac{\tau}{2})$ by \[
    \begin{cases}
        \partial_t w = f(w+z) - f(z)\\
        w(0) = u^n-z.
    \end{cases}
    \]
    \item Find $v(\tau)$ by \[
    \begin{cases}
        \partial_t v +\mathcal{A}v = f(z) - \mathcal{A}z -\partial_t z\\
        v(0) = w(\frac{\tau}{2}).
    \end{cases}
    \]
    
    \item Find $w(\frac{\tau}{2})$ by \[
    \begin{cases}
        \partial_t w = f(w+z) - f(z)\\
        w(0) = v(\tau).
    \end{cases}
    \]
    \item Define $z\in L^2(\Omega)$ such that $\left.z\right|_{\partial \Omega}=g(\cdot,t_{n+1})$.
    \item Define $u^{n+1} = w(\frac{\tau}{2}) + z$.
\end{enumerate}

Above, the introduction of the corrector $f(z)$ is independent of the solution in Step 3. Therefore, $f(z)-\mathcal{A}z-\partial_t z$ can be treated as a source term of a convection diffusion equation, which has been discussed already.
We note that for the first step, one has the flexibility to determine the interior value of $z$.
Aligned with our previous direction, we will use our choice of $\mathcal{D}^m\widetilde{g}-\widetilde{g}$ and $\mathcal{N}^m q$ for $z$, catering to the corresponding boundary conditions.

\subsection{Numerical Results}

We demonstrate the numerical experiments for time-invariant and time-variant boundary conditions. $\Omega$, $\kappa$, and $\bbeta$ are taken as in section 3. For simplicity, we assume the source term vanishes $f=0$. We compare the CD-approach and D-approach. Experiments will be run on the coarse mesh $H=\frac{1}{10},\frac{1}{20}$ and $\frac{1}{40}$ with the number of eigenfunctions $l_m=3$ and a fixed timestep $\tau = \frac{1}{10}$. Reference solutions are generated on a $200\times200$ mesh with $1000$ timesteps with the bilinear Lagrange finite element method. Without further mentioning, the error terms are recorded at the final time $T=1$.

\subsubsection{Time-invariant Boundary Conditions}

We first consider time-invariant boundary conditions:
\begin{equation}
    \begin{cases}
        u_t - \nabla\cdot(\boldsymbol{A}\nabla u) + \bbeta\cdot\nabla u = 0 & \text{ in } [0,1]^2 \times (0,1]\\
        u(x,t) = x_1^2+e^{x_1x_2} &\text{ for } (\partial [0,1]^2 )\times (0,1]\\
        u(x,0) = x_1^2 + e^{x_1 x_2} &\text{ on } [0,1]^2.
    \end{cases}
\end{equation}

Since the boundary condition is independent of time, $\mathcal{D}^m \widetilde{g}_t$ and $\widetilde{g}_t$ vanish in the numerical method. The error analysis followed directly from the time-independent problem in section 3.
\begin{table}[ht]
\centering
\begin{tabular}{c|ccccccc}
\hline\hline
 & \multicolumn{1}{c}{$N_\mathup{ov}$} & \multicolumn{1}{c}{$H$} & \multicolumn{1}{c}{$\Lambda$} & \multicolumn{1}{c}{$\|\cdot\|_{L^2}$}&\multicolumn{1}{c}{$\|\cdot\|_\mathcal{A}$}&\multicolumn{1}{c}{Time (s)} 
\\
\hline 
\multirow{3}{*}{CD-approach} & $3$ &$1/10$& 2.273414&4.54E-03& 5.90E-02&1119.079\\
& $4$ &$1/20$&2.328069&5.65E-04(12.5\%) &1.99E-02(33.8\%) &1156.276\\
& $5$ &$1/40$&3.185349 &1.17E-04(20.7\%)&9.80E-03(49.3\%)&1134.435 \\\hline
\multirow{3}{*}{D-approach} & $3$ &$1/10$& 2.273423&4.53E-03& 5.90E-02&1181.700\\
& $4$ &$1/20$&2.328069&6.31E-04(13.9\%) &2.06E-02(35.0\%)&1209.790\\
& $5$ &$1/40$&3.185349 &1.68E-04(26.6\%) &1.12E-02(54.6\%) &1140.202 \\ \hline\hline
\end{tabular}
\caption{Time-invariant Boundary. Comparison of the CD- and D-approaches with varying $H$ and $N_\mathup{ov}$ and fixed $l_m=3$, $c_\mathup{flow}=0$, $\kappa_1/\kappa_0=10^4$}
\label{tab:time_invariant_0}
\end{table}
\begin{table}[ht]
\centering
\resizebox{\textwidth}{!}{
\begin{tabular*}{\textwidth}{@{\extracolsep{\fill}}c|ccccccc}
\hline\hline
 & \multicolumn{1}{c}{$N_\mathup{ov}$} & \multicolumn{1}{c}{$H$} & \multicolumn{1}{c}{$\Lambda$} & \multicolumn{1}{c}{$\|\cdot\|_{L^2}$}&\multicolumn{1}{c}{$\|\cdot\|_\mathcal{A}$}&\multicolumn{1}{c}{Time (s)} 
\\
\hline 
\multirow{3}{*}{$\begin{array}{c}
     \text{CD-} \\
     \text{approach} 
\end{array}$} & $3$ &$1/10$& 0.5430&1.89E-03& 3.71E-02&1119.07\\
& $4$ &$1/20$&0.6613&3.71E-04(19.6\%) &1.18E-02(31.8\%) &1156.27\\
& $5$ &$1/40$&1.9085 &3.74E-05(10.1\%)&3.65E-03(30.9\%)&1134.43 \\\hline
\multirow{3}{*}{$\begin{array}{c}
     \text{D-} \\
     \text{approach} 
\end{array}$} & $3$ &$1/10$& 0.5430&2.03E-02& 6.80E-02&1181.70\\
& $4$ &$1/20$&0.6613&2.50E-02 &7.06E-02&1209.79\\
& $5$ &$1/40$&1.9085 &2.48E-02 &6.93E-02 &1140.20 \\ \hline\hline
\end{tabular*}
}
\caption{Time-invariant Boundary. Comparison of the CD- and D-approaches with varying $H$ and $N_\mathup{ov}$ and fixed $l_m=3$, $c_\mathup{flow}=4$ and $\kappa_1/\kappa_0=10^4$}
\label{tab:time_invariant_4}
\end{table}

From Table \ref{tab:time_invariant_0}, both second-order convergence in $L^2$-norm and at least first-order convergence in the energy norm can be observed for both D- and CD-approaches. Similar running times are recorded for both cases. However, since the convection term is dependent on the velocity field, the CD-approach outperforms D-approach in Table \ref{tab:time_invariant_4}. Not only does it perform relatively poorly, the D-approach also fails to achieve the convergence in $\|\cdot\|_\mathcal{A}$ with respect to $H$.

\subsubsection{Time-variant Dirichlet Boundary Conditions}

Consider the following:
\begin{equation}
    \begin{cases}
        u_t - \nabla\cdot(\kappa\nabla u) + \bbeta\cdot\nabla u = 0 & \text{ in } [0,1]^2\times [0,1]\\
        u(x,t) = (x_1^2+e^{x_1 x_2})e^{-t} &\text{ for } (\partial [0,1]^2 )\times [0,1]\\
        u(x,0) = x_1^2 + e^{x_1 x_2} &\text{ on } [0,1]^2.
    \end{cases}
\end{equation}

\begin{table}[ht]
\centering
\begin{tabular}{c|ccccccc}
\hline\hline
 & \multicolumn{1}{c}{$N_\mathup{ov}$} & \multicolumn{1}{c}{$H$} & \multicolumn{1}{c}{$\Lambda$} & \multicolumn{1}{c}{$\|\cdot\|_{L^2}$}&\multicolumn{1}{c}{$\|\cdot\|_\mathcal{A}$}&\multicolumn{1}{c}{Time (s)} 
\\
\hline 
\multirow{3}{*}{CD-approach} & $7$ &$1/10$& 2.27341&1.24E+00& 9.45E-01&5031.579
\\
& $8$ &$1/20$&2.32806&6.27E-01(50.6\%)&4.00E-01(42.3\%)&5109.764
\\
& $9$ &$1/40$&3.18534 &1.70E-01(27.1\%)&9.80E-03(40.8\%)&6793.296
 \\\hline
\multirow{3}{*}{D-approach} & $7$ &$1/10$& 2.2782&7.61E+00 & 6.91E+00 &2193.394
\\
& $8$ &$1/20$&2.33011&2.98E+00(39.2\%) &1.11E+00(16.0\%)&2234.225
\\
& $9$ &$1/40$&3.19735 &3.88E-01(13.0\%) &1.80E+00(-) &2522.632
\\ \hline\hline
\end{tabular}
\caption{Time-variant Dirichlet Boundary Condition: Relative Errors for $u^\mathup{ms}$ with different number of oversampling layers $N_\mathup{ov}$ and coarse mesh $H$, fixed $\tau = 0.1$, $l_m=3$, $\kappa_1/\kappa_0=10^4$}
\label{tab:time_variant_Dirichlet}
\end{table}

 \begin{figure}[ht]
 \begin{subfigure}[b]{0.33\textwidth}
    \centering
    \includegraphics[width=\textwidth]{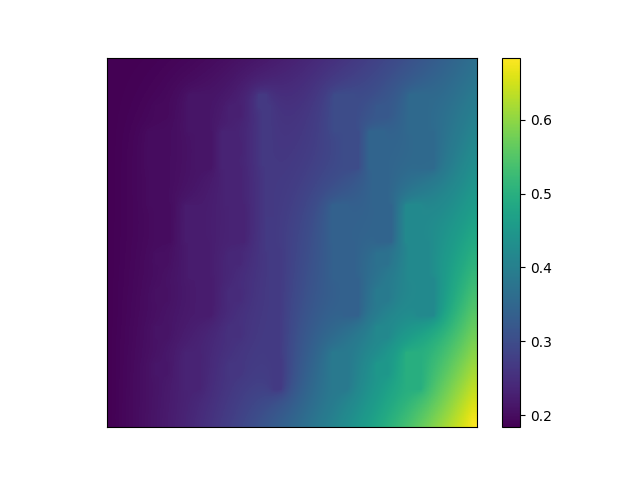}
    \caption{Reference Solution}
    \label{fig:Diri_ref_sol}
\end{subfigure}
\begin{subfigure}[b]{0.33\textwidth}
    \centering
    \includegraphics[width=\textwidth]{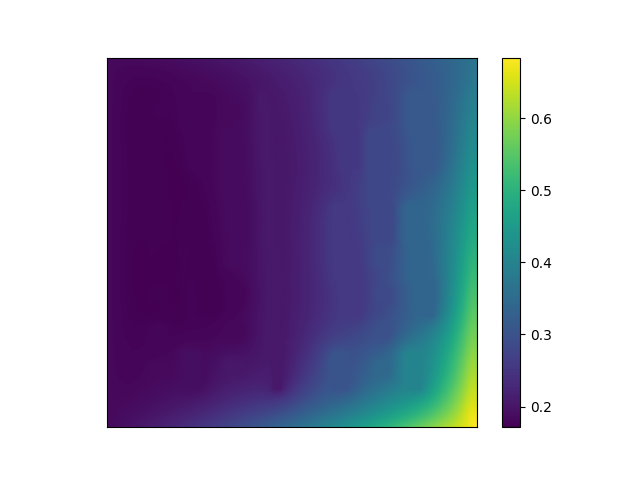}
    \caption{CD-approach}
    \label{fig:Diri_cd_sol}
\end{subfigure}
\begin{subfigure}[b]{0.33\textwidth}
    \centering
    \includegraphics[width=\textwidth]{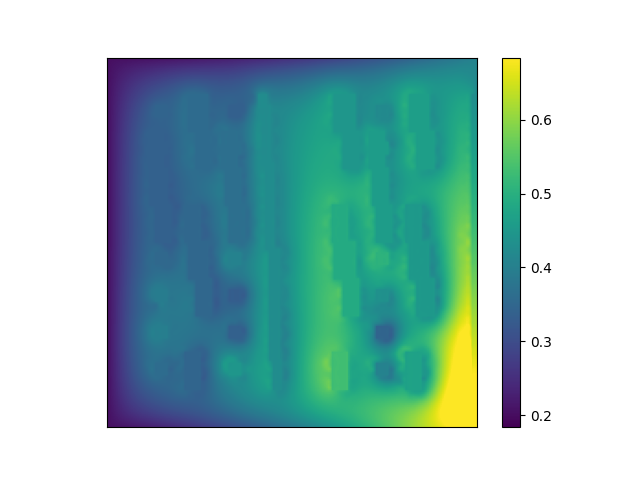}
    \caption{D-approach}
    \label{fig:Diri_d_sol}
\end{subfigure}
\caption{Solution Profiles at $T=1$ for time-variant Dirichlet IBVP, with fixed $H=\frac{1}{40}$, $l_m=3$, $\kappa_1/\kappa_0=10^4$ and $N_\mathup{ov}=9$. $\tau=\frac{1}{1000}$ for the reference solution and $\tau = \frac{1}{10}$ for the numerical solutions.}
\label{fig:diri_time_variant}
\end{figure}

As can be seen in Table \ref{tab:time_variant_Dirichlet}, the D-approach is more efficient than the CD-approach, at the cost of a higher sensitivity to the Péclet number. CD-approach is more accurate at the cost of longer computational time. Even though both present spatial convergence but the CD-approach is more robust. The time-variant case also verifies our theoretical error estimates by showing second-order (or higher) convergence in $L^2$-norm and first-order convergence in the energy norm with respect to $H$. As in Table \ref{tab:time_variant_corrector}, the increase in the number of oversampling layers does allow decay in the corrector error, and therefore verifies our theoretical convergence in $H$ as shown in Table \ref{tab:corrector_ch2}. The solution profiles can be found in Figure \ref{fig:diri_time_variant}.

\begin{table}[ht]
        \centering
        \begin{tabular}{c|ccccc}\hline\hline
             $N_\mathup{ov}$ & $2$&$3$ &$4$&$5$&$6$\\\hline
             $\|\cdot\|_\mathcal{A}$ & 3.87E-04 &1.64E-05&6.14E-06&6.13E-06&6.12E-06 \\
             $\|\cdot\|_{L^2}$ & 1.45E-03&4.37E-04&4.52E-04&4.52E-04&4.52E-04\\\hline\hline
        \end{tabular}
        \caption{Time-variant Dirichlet Boundary Condition: Relative Errors for $\mathcal{D}^m \widetilde{g}$ with different number of oversampling layers $N_\mathup{ov}$ and fixed $H=1/40$, $\tau = 0.1$, $l_m=3$, $\kappa_1/\kappa_0=10^4$}
        \label{tab:time_variant_corrector}
\end{table}

\begin{table}[ht]
    \centering
    \begin{tabular}{c|cc}\hline\hline
         $H$& $\|\cdot\|_{L^2}$ & $\|\cdot\|_{\mathcal{A}}$ \\ \hline
         $1/10$ &5.66E-03&5.48E-05\\
         $1/20$ &1.97E-03(34.8\%)&2.47E-05(45.1\%)\\
         $1/40$ &4.73E-04(24.0\%)&6.39E-06(25.9\%)\\\hline\hline
    \end{tabular}
    \caption{Time-variant Dirichlet Boundary Condition:  Errors for $\mathcal{D}^m \widetilde{g}$ with different coarse mesh $H$ and fixed $\tau = 1/10$, $\kappa_1/\kappa_0 = 10^4$, $l_m = 3$, $N_\mathup{ov} = 7$}
    \label{tab:corrector_ch2}
\end{table}
\subsubsection{Nonlinear Example}

We demonstrate combining CEM-GMsFEM with Strang Splitting to solve the following nonlinear convection-diffusion problem with a time-invariant Dirichlet boundary condition:
\begin{equation}
    \begin{cases}
        \partial_t u-\nabla \cdot(\boldsymbol{A}(x)\nabla u)+ \bbeta(x) \cdot\nabla u = u-u^3 &\text{ on }[0,1]^2\times (0,1]\\
        u = x_1^2+e^{x_1x_2} &\text{ on }(\partial [0,1]^2 )\times (0,1]\\
        u(\cdot,0) = x_1^2 + e^{x_1x_2} &\text{ on } [0,1]^2
    \end{cases}
    \label{eqn:splitting}
\end{equation}
    with $\bbeta = \bbeta_{in}$ and $c_\mathup{flow}=\frac{1}{4}$.  

 The reference solutions are generated on a $200\times 200$ mesh with $1000$ steps using the bilinear Lagrange finite element method and Backward Euler for the time discretization. Running tests on a combinations of stepsize $\tau\in \{\frac{1}{10},\frac{1}{20},\frac{1}{40}\}$, mesh size $H\in\{\frac{1}{10},\frac{1}{20},\frac{1}{40}\}$ and oversampling layers $N_\mathup{ov}\in \{7,8,9\}$. 

\begin{table}[ht]
    \centering
    \begin{tabular}{ccc|ccc}
    \hline\hline 
        $\tau$ & $N_\mathup{ov}$ & $H$ & $\Lambda$ & $\|\cdot\|_{L^2}$ & $\|\cdot\|_\mathcal{A}$\\ \hline
         $1/10$ & $7$ & $1/10$ & 2.27823 & 6.97E-02 & 7.92E-01\\
         $1/20$ & $8$ & $1/20$ & 2.33011 & 8.27E-03(11.8\%) & 1.25E-01(15.7\%) \\
         $1/40$ & $9$ & $1/40$ & 3.19735 & 4.99E-03(60.6\%) & 4.04E-02(32.3\%)\\
         \hline\hline
    \end{tabular}
    \caption{Nonlinearity case with time-invariant Dirichlet Boundary Condition, $\kappa_1/\kappa_0 = 10^4$, $l_m=3$}
    \label{tab:nonlinear}
\end{table}

As shown in Table \ref{tab:nonlinear}, the first order convergence in energy norm with respect to space and time are guaranteed. However, to achieve second order convergence in $L^2-$norm, higher oversampling layer is needed. The results could potentially be improved via adapting other temporal discretization scheme such as exponential integration \cite{POVEDA2024112796}.

\section{Conclusions}
In this paper, we propose an application of CEM-GMsFEM to solve convection diffusion equations under various types of inhomogeneous boundary conditions along with high-contrast coefficients. The method begins with constructing an auxiliary space and builds a multiscale space upon it. Boundary correctors are built upon this multiscale space by solving local oversampled cell problems. For both time independent and dependent problems, we provide convergence anaylsis and show second order convergence in $L^2$ and first order convergence in energy norm with respect to the coarse mesh, given sufficient oversampling, as agreed with numerical results. We also compare different time discretization strategies using the Backward Euler scheme. For nonlinear problems, we apply this modified method with Strang Splitting and demonstrate using a Dirichlet initial boundary value problem.

\section{Acknowledgements}
The research of Eric Chung is partially supported by the Hong Kong RGC General Research Fund (Projects: 14304021 and 14305423).

\printbibliography

@book{chung2023multiscale,

  title={Multiscale Model Reduction},

  author={Chung, Eric and Efendiev, Yalchin and Hou, Thomas Y},

  year={2023},

  publisher={Springer}

}

@article{CHUNG201669,
title = {Adaptive multiscale model reduction with Generalized Multiscale Finite Element Methods},
journal = {Journal of Computational Physics},
volume = {320},
pages = {69-95},
year = {2016},
issn = {0021-9991},
doi = {https://doi.org/10.1016/j.jcp.2016.04.054},
url = {https://www.sciencedirect.com/science/article/pii/S0021999116301097},
author = {Eric Chung and Yalchin Efendiev and Thomas Y. Hou}
}

@article{doi:10.1137/130926675,
author = {Chung, Eric T. and Efendiev, Yalchin and Leung, Wing Tat},
title = {Generalized Multiscale Finite Element Methods for Wave Propagation in Heterogeneous Media},
journal = {Multiscale Modeling \& Simulation},
volume = {12},
number = {4},
pages = {1691-1721},
year = {2014},
doi = {10.1137/130926675},

URL = { 
    
        https://doi.org/10.1137/130926675
    
    

},
eprint = { 
    
        https://doi.org/10.1137/130926675
    
    

}
}

@article{CHUNG201454,
title = {An adaptive \text{GMsFEM} for high-contrast flow problems},
journal = {Journal of Computational Physics},
volume = {273},
pages = {54-76},
year = {2014},
issn = {0021-9991},
doi = {https://doi.org/10.1016/j.jcp.2014.05.007},
url = {https://www.sciencedirect.com/science/article/pii/S002199911400343X},
author = {Eric T. Chung and Yalchin Efendiev and Guanglian Li}
}

@article{CHUNG2019298,
title = {A mixed generalized multiscale finite element method for planar linear elasticity},
journal = {Journal of Computational and Applied Mathematics},
volume = {348},
pages = {298-313},
year = {2019},
issn = {0377-0427},
doi = {https://doi.org/10.1016/j.cam.2018.08.054},
url = {https://www.sciencedirect.com/science/article/pii/S0377042718305417},
author = {Eric T. Chung and Chak Shing Lee}
}

@article{EFENDIEV2013116,
title = {Generalized multiscale finite element methods (\text{GMsFEM})},
journal = {Journal of Computational Physics},
volume = {251},
pages = {116-135},
year = {2013},
issn = {0021-9991},
doi = {https://doi.org/10.1016/j.jcp.2013.04.045},
url = {https://www.sciencedirect.com/science/article/pii/S0021999113003392},
author = {Yalchin Efendiev and Juan Galvis and Thomas Y. Hou}
}

@article{https://doi.org/10.1002/nme.5958,
author = {Chung, Eric and Efendiev, Yalchin and Leung, Wing Tat},
title = {Generalized Multiscale Finite Element Methods with energy minimizing oversampling},
journal = {International Journal for Numerical Methods in Engineering},
volume = {117},
number = {3},
pages = {316-343},
doi = {https://doi.org/10.1002/nme.5958},
url = {https://onlinelibrary.wiley.com/doi/abs/10.1002/nme.5958},
eprint = {https://onlinelibrary.wiley.com/doi/pdf/10.1002/nme.5958},
year = {2019}
}

@article{CHEN2020109133,
title = {Generalized multiscale approximation of mixed finite elements with velocity elimination for subsurface flow},
journal = {Journal of Computational Physics},
volume = {404},
pages = {109133},
year = {2020},
issn = {0021-9991},
doi = {https://doi.org/10.1016/j.jcp.2019.109133},
url = {https://www.sciencedirect.com/science/article/pii/S0021999119308381},
author = {Jie Chen and Eric T. Chung and Zhengkang He and Shuyu Sun}
}

@article{doi:10.1142/S0219876204000071,
author = {Park, Peter J. and Hou, Thomas Y.},
title = {Multiscale Numerical Methods for Singularly Perturbed Convection-Diffusion Equations},
journal = {International Journal of Computational Methods},
volume = {01},
number = {01},
pages = {17-65},
year = {2004},
doi = {10.1142/S0219876204000071},

URL = { 
    
        https://doi.org/10.1142/S0219876204000071
    
    

},
eprint = { 
    
        https://doi.org/10.1142/S0219876204000071
    
    

}
}

@article{doi:10.1137/050645646,
author = { Hughes, T. J. R. and  Sangalli, G.},
title = {Variational Multiscale Analysis: the Fine‐scale Green’s Function, Projection, Optimization, Localization, and Stabilized Methods},
journal = {SIAM Journal on Numerical Analysis},
volume = {45},
number = {2},
pages = {539-557},
year = {2007},
doi = {10.1137/050645646},

URL = { 
    
        https://doi.org/10.1137/050645646
    
    

},
eprint = { 
    
        https://doi.org/10.1137/050645646
    
    

}
}

@article{JOHN20064594,
title = {A two-level variational multiscale method for convection-dominated convection–diffusion equations},
journal = {Computer Methods in Applied Mechanics and Engineering},
volume = {195},
number = {33},
pages = {4594-4603},
year = {2006},
issn = {0045-7825},
doi = {https://doi.org/10.1016/j.cma.2005.10.006},
url = {https://www.sciencedirect.com/science/article/pii/S0045782505004457},
author = {Volker John and Songul Kaya and William Layton}
}

@article{SONG20102226,
title = {A variational multiscale method based on bubble functions for convection-dominated convection–diffusion equation},
journal = {Applied Mathematics and Computation},
volume = {217},
number = {5},
pages = {2226-2237},
year = {2010},
issn = {0096-3003},
doi = {https://doi.org/10.1016/j.amc.2010.07.023},
url = {https://www.sciencedirect.com/science/article/pii/S009630031000768X},
author = {Lina Song and Yanren Hou and Haibiao Zheng}
}

@article{XIE2021107077,
title = {Variational multiscale virtual element method for the convection-dominated diffusion problem},
journal = {Applied Mathematics Letters},
volume = {117},
pages = {107077},
year = {2021},
issn = {0893-9659},
doi = {https://doi.org/10.1016/j.aml.2021.107077},
url = {https://www.sciencedirect.com/science/article/pii/S0893965921000355},
author = {Cong Xie and Gang Wang and Xinlong Feng}
}

@article{KIM20142251,
title = {A multiscale discontinuous Galerkin method for convection–diffusion–reaction problems},
journal = {Computers \& Mathematics with Applications},
volume = {68},
number = {12, Part B},
pages = {2251-2261},
year = {2014},
note = {Advances in Computational Partial Differential Equations},
issn = {0898-1221},
doi = {https://doi.org/10.1016/j.camwa.2014.08.007},
url = {https://www.sciencedirect.com/science/article/pii/S0898122114003678},
author = {Mi-Young Kim and Mary F. Wheeler}
}

@article{Chung_Leung_2013, title={A Sub-Grid Structure Enhanced Discontinuous Galerkin Method for Multiscale Diffusion and Convection-Diffusion Problems}, volume={14}, DOI={10.4208/cicp.071211.070912a}, number={2}, journal={Communications in Computational Physics}, author={Chung, Eric T. and Leung, Wing Tat}, year={2013}, pages={370–392}}

@misc{chung2018multiscale,
      title={Multiscale stabilization for convection diffusion equations with heterogeneous velocity and diffusion coefficients}, 
      author={Eric T. Chung and Yalchin Efendiev and Wing Tat Leung},
      year={2018},
      eprint={1807.11529},
      archivePrefix={arXiv},
      primaryClass={math.NA}
}

@article{CALO2016359,
title = {Multiscale stabilization for convection-dominated diffusion in heterogeneous media},
journal = {Computer Methods in Applied Mechanics and Engineering},
volume = {304},
pages = {359-377},
year = {2016},
issn = {0045-7825},
doi = {https://doi.org/10.1016/j.cma.2016.02.014},
url = {https://www.sciencedirect.com/science/article/pii/S0045782516300445},
author = {Victor M. Calo and Eric T. Chung and Yalchin Efendiev and Wing Tat Leung}
}

@misc{cheung2018constraint,
      title={Constraint Energy Minimizing Generalized Multiscale Finite Element Method for dual continuum model}, 
      author={Siu Wun Cheung and Eric T. Chung and Yalchin Efendiev and Wing Tat Leung and Maria Vasilyeva},
      year={2018},
      eprint={1807.10955},
      archivePrefix={arXiv},
      primaryClass={math.NA}
}

@article{CHEUNG2020112960,
title = {Constraint energy minimizing generalized multiscale discontinuous Galerkin method},
journal = {Journal of Computational and Applied Mathematics},
volume = {380},
pages = {112960},
year = {2020},
issn = {0377-0427},
doi = {https://doi.org/10.1016/j.cam.2020.112960},
url = {https://www.sciencedirect.com/science/article/pii/S037704272030251X},
author = {Siu Wun Cheung and Eric T. Chung and Wing Tat Leung}
}

@article{Chung_2018,
   title={Constraint energy minimizing generalized multiscale finite element method in the mixed formulation},
   volume={22},
   ISSN={1573-1499},
   url={http://dx.doi.org/10.1007/s10596-018-9719-7},
   DOI={10.1007/s10596-018-9719-7},
   number={3},
   journal={Computational Geosciences},
   publisher={Springer Science and Business Media LLC},
   author={Chung, Eric and Efendiev, Yalchin and Leung, Wing Tat},
   year={2018},
   month=feb, pages={677–693} }

@article{FU2020109569,
title = {Constraint energy minimizing generalized multiscale finite element method for nonlinear poroelasticity and elasticity},
journal = {Journal of Computational Physics},
volume = {417},
pages = {109569},
year = {2020},
issn = {0021-9991},
doi = {https://doi.org/10.1016/j.jcp.2020.109569},
url = {https://www.sciencedirect.com/science/article/pii/S0021999120303430},
author = {Shubin Fu and Eric Chung and Tina Mai}
}

@article{Fu2018ConstraintEM,
  title={Constraint Energy Minimizing Generalized Multiscale Finite Element Method for high-contrast linear elasticity problem},
  author={Shubin Fu and Eric T. Chung},
  journal={arXiv: Numerical Analysis},
  year={2018},
  url={https://api.semanticscholar.org/CorpusID:209888352}
}

@article{doi:10.1137/18M1193128,
author = {Li, Mengnan and Chung, Eric and Jiang, Lijian},
title = {A Constraint Energy Minimizing Generalized Multiscale Finite Element Method for Parabolic Equations},
journal = {Multiscale Modeling \& Simulation},
volume = {17},
number = {3},
pages = {996-1018},
year = {2019},
doi = {10.1137/18M1193128},

URL = { 
    
        https://doi.org/10.1137/18M1193128
    
    

},
eprint = { 
    
        https://doi.org/10.1137/18M1193128
    
    

}
}

@article{VASILYEVA2019660,
title = {Constrained energy minimization based upscaling for coupled flow and mechanics},
journal = {Journal of Computational Physics},
volume = {376},
pages = {660-674},
year = {2019},
issn = {0021-9991},
doi = {https://doi.org/10.1016/j.jcp.2018.09.054},
url = {https://www.sciencedirect.com/science/article/pii/S0021999118306569},
author = {Maria Vasilyeva and Eric T. Chung and Yalchin Efendiev and Jihoon Kim}
}

@article{CHUNG2018298,
title = {Constraint Energy Minimizing Generalized Multiscale Finite Element Method},
journal = {Computer Methods in Applied Mechanics and Engineering},
volume = {339},
pages = {298-319},
year = {2018},
issn = {0045-7825},
doi = {https://doi.org/10.1016/j.cma.2018.04.010},
url = {https://www.sciencedirect.com/science/article/pii/S0045782518301804},
author = {Eric T. Chung and Yalchin Efendiev and Wing Tat Leung}
}

@article{10.1093/imanum/drx047,
    author = {Alonso-Mallo, I and Cano, B and Reguera, N},
    title = "{Avoiding order reduction when integrating linear initial boundary value problems with exponential splitting methods}",
    journal = {IMA Journal of Numerical Analysis},
    volume = {38},
    number = {3},
    pages = {1294-1323},
    year = {2017},
    month = {08},
    issn = {0272-4979},
    doi = {10.1093/imanum/drx047},
    url = {https://doi.org/10.1093/imanum/drx047},
    eprint = {https://academic.oup.com/imajna/article-pdf/38/3/1294/25170894/drx047.pdf},
}

@article{CONNORS2014181,
title = {Quantification of errors for operator-split advection–diffusion calculations},
journal = {Computer Methods in Applied Mechanics and Engineering},
volume = {272},
pages = {181-197},
year = {2014},
issn = {0045-7825},
doi = {https://doi.org/10.1016/j.cma.2014.01.005},
url = {https://www.sciencedirect.com/science/article/pii/S0045782514000103},
author = {Jeffrey M. Connors and Jeffrey W. Banks and Jeffrey A. Hittinger and Carol S. Woodward}
}

@article{LeVeque1981NumericalMB,
  title={Numerical methods based on additive splittings for hyperbolic partial differential equations},
  author={Randall J. LeVeque and Joseph E. Oliger},
  journal={Mathematics of Computation},
  year={1981},
  volume={40},
  pages={469-497},
  url={https://api.semanticscholar.org/CorpusID:39749379}
}

@inproceedings{10.5555/166337.166347,
author = {Dawson, Clint N. and Wheeler, Mary F.},
title = {Time-splitting methods for advection-diffusion-reaction equations arising in contaminant transport},
year = {1992},
isbn = {0898713021},
publisher = {Society for Industrial and Applied Mathematics},
address = {USA},
booktitle = {Proceedings of the Second International Conference on Industrial and Applied Mathematics},
pages = {71–82},
numpages = {12},
location = {Washington, D.C., USA},
series = {ICIAM 91}
}

@article{splitting,
author = {Descombes, Stéphane},
year = {2001},
month = {10},
pages = {1481-1501},
title = {Convergence of a splitting method of high order for reaction-diffusion systems},
volume = {70},
journal = {Math. Comput.},
doi = {10.1090/S0025-5718-00-01277-1}
}

@article{GERISCH2002159,
title = {Operator splitting and approximate factorization for taxis–diffusion–\linebreak reaction models},
journal = {Applied Numerical Mathematics},
volume = {42},
number = {1},
pages = {159-176},
year = {2002},
note = {Numerical Solution of Differential and Differential-Algebraic Equations, 4-9 September 2000, Halle, Germany},
issn = {0168-9274},
doi = {https://doi.org/10.1016/S0168-9274(01)00148-9},
url = {https://www.sciencedirect.com/science/article/pii/S0168927401001489},
author = {A. Gerisch and J.G. Verwer}
}

@article{HUNDSDORFER1995191,
title = {A note on splitting errors for advection-reaction equations},
journal = {Applied Numerical Mathematics},
volume = {18},
number = {1},
pages = {191-199},
year = {1995},
issn = {0168-9274},
doi = {https://doi.org/10.1016/0168-9274(95)00069-7},
url = {https://www.sciencedirect.com/science/article/pii/0168927495000697},
author = {W. Hundsdorfer and J.G. Verwer}
}

@article{doi:10.1137/140994204,
author = {Einkemmer, Lukas and Ostermann, Alexander},
title = {Overcoming Order Reduction in Diffusion-Reaction Splitting. Part 1: Dirichlet Boundary Conditions},
journal = {SIAM Journal on Scientific Computing},
volume = {37},
number = {3},
pages = {A1577-A1592},
year = {2015},
doi = {10.1137/140994204},

URL = { 
    
        https://doi.org/10.1137/140994204
    
    

},
eprint = { 
    
        https://doi.org/10.1137/140994204
    
    

}
}

@article{doi:10.1137/16M1056250,
author = {Einkemmer, Lukas and Ostermann, Alexander},
title = {Overcoming Order Reduction in Diffusion-Reaction Splitting. Part 2: Oblique Boundary Conditions},
journal = {SIAM Journal on Scientific Computing},
volume = {38},
number = {6},
pages = {A3741-A3757},
year = {2016},
doi = {10.1137/16M1056250},

URL = { 
    
        https://doi.org/10.1137/16M1056250
    
    

},
eprint = { 
    
        https://doi.org/10.1137/16M1056250
    
    

}
}

@article{Einkemmer_2018,
publisher = {Cornell University Library, arXiv.org},
title = {A comparison of boundary correction methods for Strang splitting},
year = {2018},
author = {Einkemmer, Lukas and Ostermann, Alexander},
address = {Ithaca},
copyright = {2018. This work is published under http://arxiv.org/licenses/nonexclusive-distrib/1.0/ (the “License”). Notwithstanding the ProQuest Terms and Conditions, you may use this content in accordance with the terms of the License.},
issn = {2331-8422},
journal = {arXiv.org},
language = {eng},
}

@article{doi:10.1137/19M1257081,
author = {Bertoli, Guillaume and Vilmart, Gilles},
title = {Strang Splitting Method for Semilinear Parabolic Problems With Inhomogeneous Boundary Conditions: A Correction Based on the Flow of the Nonlinearity},
journal = {SIAM Journal on Scientific Computing},
volume = {42},
number = {3},
pages = {A1913-A1934},
year = {2020},
doi = {10.1137/19M1257081},

URL = { 
    
        https://doi.org/10.1137/19M1257081
    
    

},
eprint = { 
    
        https://doi.org/10.1137/19M1257081
    
    

}
}

@article{POVEDA2024112796,
title = {A second-order exponential integration constraint energy minimizing generalized multiscale method for parabolic problems},
journal = {Journal of Computational Physics},
volume = {502},
pages = {112796},
year = {2024},
issn = {0021-9991},
doi = {https://doi.org/10.1016/j.jcp.2024.112796},
url = {https://www.sciencedirect.com/science/article/pii/S0021999124000457},
author = {Leonardo A. Poveda and Juan Galvis and Eric Chung}
}

@article{doi:10.1137/21M1459113,
author = {Ye, Changqing and Chung, Eric T.},
title = {Constraint Energy Minimizing Generalized Multiscale Finite Element Method for Inhomogeneous Boundary Value Problems with High Contrast Coefficients},
journal = {Multiscale Modeling \& Simulation},
volume = {21},
number = {1},
pages = {194-217},
year = {2023},
doi = {10.1137/21M1459113},

URL = { 
    
        https://doi.org/10.1137/21M1459113
    
    

},
eprint = { 
    
        https://doi.org/10.1137/21M1459113
    
    

}
}
\end{document}